\documentclass[12pt]{amsart}
\usepackage{xypic}
\usepackage{amssymb}
\usepackage{amsthm}
\usepackage{bm}
\usepackage{latexsym}
\usepackage{amsmath}
\usepackage{eufrak}
\usepackage{mathrsfs}
\usepackage{amscd}
\usepackage[all]{xy}
\usepackage[usenames]{color}
\usepackage[dvips]{graphicx}
\usepackage{color}
\usepackage{amscd}
\theoremstyle{plain}
\usepackage[usenames,dvipsnames]{pstricks}
\usepackage{epsfig}
\usepackage{pst-grad} 
\usepackage{pst-plot} 
\usepackage{longtable}

\newtheorem{theorem}{Theorem}[section]
\newtheorem{proposition}[theorem]{Proposition}
\newtheorem{lemma}[theorem]{Lemma}
\newtheorem{corollary}[theorem]{Corollary}

\theoremstyle{definition}

\newcommand{\appsection}[1]{\let\oldthesection\thesection
\renewcommand{\thesection}{Appendix \oldthesection}
\section{#1}\let\thesection\oldthesection}

\newtheorem{definition}[theorem]{Definition}

\theoremstyle{remark}

\newtheorem{remark}[theorem]{Remark}
\newtheorem{example}[theorem]{Example}

\DeclareMathOperator{\disc}{discr}

\def\D{{\mathbb{D}}}

\def\Q{{\mathbb{Q}}}
\def\C{{\mathbb{C}}}
\def\P{{\mathbb{P}}}

\def\E{{\mathbb{E}}}

\def\O{{\mathcal{O}}}

\def\X{{\mathcal{X}}}
\def\Y{{\mathcal{Y}}}
\def\M{{\mathcal{M}}}

\def\W{{\mathcal{W}}}

\def\DD{{\mathcal{D}}}

\DeclareMathOperator{\Def}{Def} \DeclareMathOperator{\QG}{QG}

\newcommand{\eni}{mk1A}
\newcommand{\enii}{mk2A}

\begin{document}
\bibliographystyle{amsplain}
\title[Construcci\'on]{Identifying neighbors of stable surfaces}
\author{\textrm{Giancarlo Urz\'ua}}
\date{\today}

\email{urzua@mat.puc.cl}

\maketitle

\begin{abstract}
We identify the stable surfaces around the stable limit of the examples of Y. Lee and J. Park \cite{LP07}, and H. Park, J. Park and D. Shin \cite{PPS09} using the explicit $3$-fold Mori theory in \cite{HTU12}. These surfaces belong to the Koll\'ar--Shepherd-Barron--Alexeev compactification of the moduli space of simply connected surfaces of general type with $p_g=0$ and $K^2=1,2,3$.
\end{abstract}

\tableofcontents

\section{Introduction} \label{intro}

A main application of \cite{HTU12} is to have an explicit $3$-fold Mori theory to find stable limits of $\Q$-Gorenstein one parameter degenerations of surfaces with only log terminal singularities. The aim of this paper is to run \cite[\S5]{HTU12} on the singular examples of Y. Lee and J. Park \cite{LP07}, and H. Park, J. Park and D. Shin \cite{PPS09} to identify stable surfaces around them. These surfaces belong to the Koll\'ar--Shepherd-Barron--Alexeev (KSBA) compactification of the moduli space of (simply connected) surfaces of general type with $p_g=0$ and $K^2=1,2,3$ \cite{KSB88,AM04,K90}. This moduli space has no explicit description for any $K^2$. It is not even known whether it is irreducible. Moreover, the only explicit surfaces with those invariants are Barlow surfaces \cite[VII.10]{BHPV04} \footnote{Conjecturally we also have the Craighero-Gattazzo surface.}, where $K^2=1$, and for the rest we only know existence via the $\Q$-Gorenstein smoothing method pioneered in \cite{LP07}.

We work out one example for each $K^2$, and state results for the others. We find their stable (KSBA) models (see Lemma \ref{l2} for the general picture), and the smooth minimal model of the stable singular surfaces around them. Lee-Park examples represent points of the moduli space of stable surfaces \footnote{ A local model of the coarse moduli space at these surfaces is the space of $\Q$-Gorenstein deformations, which is smooth at all of the Lee-Park examples \cite[\S3]{H2011}, modulo a finite group of automorphisms.}, with local dimension $10-2K^2$, and each of its Wahl singularities $\frac{1}{n^2}(1,na-1)$ defines a boundary divisor $\DD{n \choose a}$. In this way, we will be identifying general points on these divisors. This is done in \S\ref{s1}, \S\ref{s2}, and \S\ref{s3}.

In \S\ref{s0}, we summarize the results we need from \cite{HTU12}, passing through the necessary notations and facts. Then, in \S\ref{identify} we describe in detail the strategy to identify stable surfaces around a given one. We would like to remark that the techniques used in \S\ref{identify} can be applied to surfaces with other invariants. The choice of invariants in this paper reflects the interest of the author.

Before working out the examples, in \S\ref{app1} we describe how $p_g=0$ elliptic surfaces can be constructed via $\Q$-Gorenstein smoothings. Apart from putting these elliptic surfaces in perspective with the general type constructions, this description will be used in the next sections to identify stable surfaces.

The identification in \S\ref{s1}, \S\ref{s2}, and \S\ref{s3} shows the presence of various special surfaces in the KSBA boundary. For example, there are singular stable surfaces whose smooth minimal models are $p_g=0$ surfaces of general type which contain certain configurations of curves. There are also stable surfaces whose smooth minimal models are Dolgachev surfaces (i.e., simply connected elliptic fibrations with $p_g=0$ and Kodaira dimension $1$, see Corollary \ref{dolga}), and special rational surfaces. In some cases, these rational examples are distinct from the type of examples in \cite{LP07,PPS09} and related papers, where the construction depends on rational elliptic fibrations with certain singular fibers. Hence this brings a new type of construction; see \cite{Urz4} for concrete examples.

Finally some conventions. We write the same letter to denote a curve and its strict transform under a birational map. We use Kodaira's notation \cite[p.201]{BHPV04} for singular fibers of elliptic fibrations. A $(-n)$-curve in a smooth surface is a curve $C \simeq \P^1$ with $C^2=-n$. The symbol $\D$ will be used for a smooth analytic germ of a curve. A surface in the Koll\'ar--Shepherd-Barron--Alexeev moduli space will be called either stable or KSBA surface. The ground field is $\C$.

\subsection*{Acknowledgements}

I am grateful to the anonymous referee for very helpful suggestions which have significatively improved the presentation of the paper. I have also benefited from many conversations with Paul Hacking and Jenia Tevelev. I was supported by the FONDECYT Inicio grant 11110047 funded by the Chilean Government.

\section{Preliminaries} \label{s0}

The purpose of this section is to give a summary of some results from \cite{HTU12} which will be used in the next sections. We first recall some terminology and facts from various sources.

\def\bQ{\Bbb Q}
\def\bC{\Bbb C}

\subsection{Cyclic quotient singularities} \label{cyclic}

A cyclic quotient singularity $Y$, denoted by $\frac{1}{m}(1,q)$, is a germ at the origin of the quotient of $\C^2$ by the action of
$\mu_m$ given by $(x,y)\mapsto (\mu x, \mu^q y)$, where $\mu$ is a primitive $m$-th root of $1$, and $q$ is an integer with $0<q<m$ and gcd$(q,m)=1$; cf. \cite[III \S5]{BHPV04}. Let $\sigma \colon \widetilde{Y} \rightarrow Y$ be the minimal resolution of $Y$. Figure \ref{exdiv} shows the exceptional curves $E_i=\P^1$ of $\sigma$, for $1 \leq i \leq s$, and the strict transforms $E_0$ and $E_{s+1}$ of $(y=0)$ and $(x=0)$ respectively.

\begin{figure}[htbp]
\includegraphics[width=11.5cm]{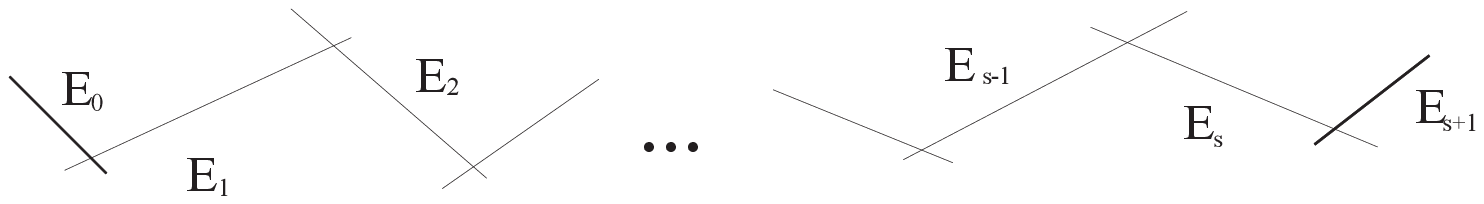}
\caption{Exceptional divisors over $\frac{1}{m}(1,q)$, $E_0$ and
$E_{s+1}$} \label{exdiv}
\end{figure}

The numbers $E_i^2=-b_i$ are computed using the {\em Hirzebruch-Jung continued fraction}
$$ \frac{m}{q} = b_1 - \frac{1}{b_2 - \frac{1}{\ddots - \frac{1}{b_s}}} =: [b_1, \ldots ,b_s].$$

A configuration of curves $[b_1,\ldots,b_s]$ in a nonsingular surface will mean the corresponding exceptional divisor of the singularity $\frac{1}{m}(1,q)$.

We use same notation for continued fractions $[b_1, \ldots ,b_s]$ even when some $b_i$ are $1$. This will happen in \S \ref{enbhd} for example.

The continued fraction $[b_1,\ldots,b_s]$ defines the sequence of integers $$ 0=\beta_{s+1} < 1=\beta_s < \ldots < q=\beta_1 < m= \beta_0 $$ where $\beta_{i+1}= b_{i}\beta_i - \beta_{i-1}$. In this way, $\frac{\beta_{i-1}}{\beta_{i}}=[b_i,\ldots,b_s]$. Partial fractions $\frac{\alpha_i}{\gamma_i} =[b_1,\ldots,b_{i-1}]$ are computed through the sequences $$ 0=\alpha_0 < 1=\alpha_1 < \ldots < q^{-1}=\alpha_s < m= \alpha_{s+1},$$ where $\alpha_{i+1}=b_i\alpha_{i} - \alpha_{i-1}$ ($q^{-1}$ is the integer such that $0<q^{-1}<m$ and $q q^{-1} \equiv 1 ($mod $m)$), and $\gamma_0=-1$, $\gamma_1=0$, $\gamma_{i+1}=b_i \gamma_i - \gamma_{i-1}$. We have $\alpha_{i+1}\gamma_i - \alpha_i \gamma_{i+1}=-1$, $\beta_i = q \alpha_i - m \gamma_i$, and $\frac{m}{q^{-1}}=[b_s,\ldots,b_1]$. These numbers appear in the pull-back formulas $$ \sigma^*\big((y=0)\big) = \sum_{i=0}^{s+1} \frac{\beta_i}{m} E_i, \ \ \ \text{and} \ \ \ \sigma^*\big((x=0)\big)= \sum_{i=0}^{s+1} \frac{\alpha_i}{m} E_i,$$ and $K_{\widetilde{Y}}
\equiv \sigma^*(K_Y) + \sum_{i=1}^s (-1 +\frac{\beta_i+\alpha_i}{m}) E_i$.

\subsection{$\Q$-Gorenstein deformations} \label{Qgd}

The following terminology and facts are from~\cite{KSB88}.

\begin{definition}
Let $Y$ be a normal surface with only quotient singularities, and let $\D$ be a smooth analytic germ of a curve. A deformation $(Y \subset \Y)
\rightarrow (0 \in \D)$ of $Y$ is called a {\em smoothing} if its general fiber is smooth. It is {\em $\Q$-Gorenstein} if $K_{\Y}$
is $\Q$-Cartier. \label{def}
\end{definition}

\begin{definition}
A germ of a normal surface $Y$ is called a {\em T-singularity} if it is a quotient singularity and admits a $\bQ$-Gorenstein smoothing.  \label{T-sing}
\end{definition}

A $T$-singularity is either a du Val singularity or a cyclic quotient singularity ${1\over dn^2}(1,dna-1)$ with gcd$(n,a)=1$ \cite[Prop.3.10]{KSB88}. A $T$-singularity with a one-dimensional $\bQ$-Gorenstein versal deformation space is either a node $\text{A}_1$ or a {\em Wahl singularity} ${1\over n^2}(1,na-1)$.

\begin{definition}
Let $(Q\in Y)$ be a germ of a two dimensional quotient singularity. A proper birational map $f \colon X \to Y$ is called a {\em $P$-resolution} if $f$ is an isomorphism away from~$Q$, $X$ has $T$-singularities only, and $K_X$ is ample relative to~$f$ \cite[Def.3.8]{KSB88}. \label{P-res}
\end{definition}

By \cite[3.9]{KSB88}, there is a natural bijection between P-resolutions $X^+\to Y$ and irreducible components of the formal
deformation space $\Def(Y)$. Namely, let $\Def^{\QG}(X^+)$ denote the versal $\Q$-Gorenstein deformation space of $X^+$. Recall that
for any rational surface singularity $Z$ and its partial resolution $X\to Z$, there is an induced map $\Def(X) \to \Def(Z)$ of
formal deformation spaces \cite[1.4]{Wahl76}, which we refer to as {\em blowing down deformations}. In particular, we have a map
$\Def^{\QG}(X^+) \rightarrow \Def(Y)$. The germ $\Def^{\QG}(X^+)$ is smooth, the map $\Def^{\QG}(X^+) \rightarrow \Def(Y)$ is a closed embedding,
and it identifies $\Def^{\QG}(X^+)$ with an irreducible component of $\Def(Y)$. All irreducible components of $\Def(Y)$ arise in
this fashion (in a unique way).

\subsection{Extremal neighborhoods} \label{enei}

Now some definitions from \cite{KM92}.

\begin{definition}
An {\em extremal neighborhood} $(C \subset \X) \to (Q\in \Y)$ is a proper birational morphism between normal $3$-folds $F \colon \X \to \Y$ such that
\begin{enumerate}
\item The canonical class $K_{\X}$ is $\Q$-Cartier and $\X$ has only terminal singularities.

\item There is a distinguished point $Q\in \Y$ such that ${F}^{-1}(Q)$ consists of an irreducible curve $C \subset \X$.

\item $K_{\X} \cdot C <0$.
\end{enumerate}
\label{ExtrNbhd}
\end{definition}

Let Exc$(F)$ be the exceptional loci of $F$. An extremal neighborhood is {\em flipping} if Exc$(F)=C$. Otherwise, Exc$(F)$ is two dimensional, and $F$ is called {\em divisorial}.

In the flipping case, $K_{\Y}$ is not $\Q$-Cartier. Then one attempts another type of birational modification. A {\em flip} of a flipping
extremal neighborhood $$F \colon (C \subset \X) \to (Q \in \Y)$$ is a proper birational morphism $$F^+ \colon (C^+ \subset \X^+) \to (Q \in
\Y)$$ where $\X^+$ is normal with terminal singularities, Exc$(F^+)=C^+$ is a curve, and $K_{\X^+}$ is $\Q$-Cartier and
$F^+$-ample. A flip induces a birational map $\X \dashrightarrow \X^+$ to which we also refer as flip. When a flip exists then it
is unique (cf. \cite{KM1998}). Mori \cite{Mori88} proves that ($3$-fold) flips always exist.

\subsection{Explicit semi-stable MMP} \label{enbhd}

In \cite{HTU12} we focus on two particular types of extremal neighborhoods, which appear naturally when working on the Koll\'ar--Shepherd-Barron--Alexeev
compactification of the moduli of surfaces of general type.

\begin{definition}
Let $(Q \in Y)$ be a two dimensional cyclic quotient singularity germ. Assume there is a partial resolution $f \colon X \to Y$ of $Y$ such that $f^{-1}(Q)$ is a smooth rational curve $C$ with one (two) Wahl singularity(ies) of $X$ on it. Suppose $K_{X} \cdot C<0$. Let $(X \subset \X) \rightarrow (0 \in \D)$ be a $\Q$-Gorenstein smoothing of $X$ over a smooth analytic germ of a curve $\D$. Let $(Y \subset \Y) \rightarrow (0 \in \D)$ be the corresponding blowing down deformation of $Y$. The induced birational morphism $(C \subset \X) \rightarrow (Q \in \Y)$ is called \textit{extremal neighborhood of type \eni ~(\enii)}; we denote it by \eni ~(\enii).
\label{EN}
\end{definition}

These extremal neighborhoods are of type k1A and k2A (cf. \cite{KM92,Mori02}), and they are minimal with respect to the second
betti number, which is equal to $1$, of the Milnor fiber of $(Y \subset \Y) \rightarrow (0 \in \D)$ (see \cite[Prop.2.1]{HTU12} for more discussion on this).

\begin{definition}
A P-resolution $f^+ \colon X^+ \to Y$ of a two dimensional cyclic quotient singularity germ $(Q \in Y)$ is called \textit{extremal
P-resolution} if ${f^+}^{-1}(Q)$ is a smooth rational curve $C^+$, and $X^+$ has only Wahl singularities (thus at most two; cf. \cite[Lemma 3.14]{KSB88}). \label{EPR}
\end{definition}

\begin{proposition}
Let $(C \subset \X) \to (Q \in \Y)$ be a flipping ~\eni ~or \enii, where $(C \subset X) \to (Q \in Y)$ is the contraction of $C$ between the special fibers. Then there exists an extremal P-resolution $(C^+ \subset X^+) \to (Q \in Y)$, such that the flip $(C^+ \subset \X^+) \to (Q \in \Y)$ is obtained by the blowing down deformation of a $\Q$-Gorenstein smoothing of $X^+$. The commutative diagram of maps is $$ \xymatrix{ (C \subset \X) \ar[ddr] \ar[dr] \ar@{-->}[rr]^{\text{flip}} &  & (C^+ \subset \X^+) \ar[dl] \ar[ldd] \\ & (Q \in \Y) \ar[d] & \\ & (0 \in \D), & }$$ and restricted to the special fibers we have
 $$ \xymatrix{ (C \subset X)  \ar[dr] \ar@{-->}[rr] &  & (C^+ \subset X^+) \ar[dl] \\ & (Q \in Y). & } $$
\label{flipPres}
\end{proposition}

\begin{proof}
\cite[Sect.11 and Thm.13.5]{KM92}. (See \cite{Mori02,HTU12} for explicit equations of the surfaces and $3$-folds involved.)
\end{proof}

\begin{proposition}
If an \eni ~or \enii ~is divisorial, then $(Q \in Y)$ is a Wahl singularity. The divisorial contraction $\X \rightarrow \Y$ induces the blowing down of a $(-1)$-curve between the smooth fibers of $\X \to \D$ and $\Y \to \D$. \label{divtype}
\end{proposition}

\begin{proof}
Since $K_{X} \cdot C<0$ and $X$ has only Wahl singularities, the divisorial contraction $\X \to \Y$ induces the blowing down of a
$(-1)$-curve between the smooth fibers of $\X \to \D$ and $\Y \to \D$; see \cite[Prop.3.16(b)]{HTU12}. Since it is the divisorial contraction of an extremal ray, the $3$-fold $\Y$ is $\Q$-Gorenstein, and so $(Q \in Y)$ is a T-singularity. If $X'$ and $Y'$ are smooth fibers of $\X
\to \D$ and $Y \to \D$, then $K_{X}^2=K_{X'}^2=K_{Y'}^2-1=K_{Y}^2-1$. Hence, since the second Betti number of the smoothing $(X \subset \X) \to (0 \in \D)$ is one, we have that the Milnor number of the smoothing $\Y \to \D$ of $(Q \in Y)$ is zero. A T-singularity with a smoothing which has Milnor number zero is a Wahl singularity (see for example \cite[Lemma 2.4]{HP2010}).
\end{proof}

{The following is the \underline{numerical description} of the $X$ in an {\eni} ~or in an {\enii} (Definition \ref{EN}), and of the $X^+$ in an extremal P-resolution (Definition \ref{EPR}). This description only requires toric computations on surfaces, the $3$-folds $\X$ and $\X^+$ do not play a role. See more details in \cite[\S2]{HTU12}.

\vspace{0.4cm}

\underline{($X \to Y$ for \eni)}: Fix an \eni ~with Wahl singularity $\frac{1}{m^2}(1,ma-1)$. Let $\frac{m^2}{ma-1}=[e_1,\ldots,e_s]$ be its continued fraction. Let $E_1,\ldots,E_s$ be the exceptional curves of the minimal resolution $\widetilde{X}$ of $X$ with $E_j^2=-e_j$ for all $j$. Notice that $K_{X} \cdot C <0$ and $C \cdot C <0$ imply that the
strict transform of $C$ in $\widetilde{X}$ is a $(-1)$-curve intersecting only one component $E_i$ transversally at one point. This data will be written as $$[e_1,\ldots,\overline{e_i},\ldots,e_s]$$ so that $\frac{\Delta}{\Omega} = [e_1,\ldots,e_i-1,\dots,e_s]$ where $0<\Omega <\Delta$, and $(Q \in Y)$ is $\frac{1}{\Delta}(1,\Omega)$. Let $\beta_i,\alpha_i,\gamma_i$ be the numbers defined in \S \ref{cyclic} for the singularity $\frac{1}{m^2}(1,ma-1)$. Then $$\Delta= m^2 - \beta_i \alpha_i \ \ \ \ \ \ \Omega= ma-1-\gamma_i \beta_i$$ and, if
$\delta:= \frac{\beta_i+\alpha_i}{m}$, we have $K_{X} \cdot C=\frac{- \delta}{m} <0$ and $C \cdot C= \frac{- \Delta}{m^2}<0$.

\vspace{0.4cm}

\underline{($X \to Y$ for \enii)}: Consider now an \enii ~with Wahl singularities $\frac{1}{m_j^2}(1,m_j a_j -1)$ ($j=1,2$). Let $E_1,\ldots,E_{s_1}$
and $F_1,\ldots,F_{s_2}$ be the exceptional divisors over $\frac{1}{m_1^2}(1,m_1 a_1 -1)$ and $\frac{1}{m_2^2}(1,m_2 a_2
-1)$ respectively, such that $\frac{m_1^2}{m_1 a_1-1}=[e_1,\ldots,e_{s_1}]$ and $\frac{m_2^2}{m_2 a_2 -1}=[f_1,\ldots,f_{s_2}]$ with $E_i^2=-e_i$ and $F_j^2=-f_j$. We know that the strict transform of $C$ in the minimal resolution $\widetilde{X}$ of $X$ is a $(-1)$-curve intersecting only one $E_i$ and one $F_j$ transversally at one point, and these two exceptional curves are at the ends of these exceptional chains. The data for \enii ~will be written as $$[f_{s_2},\ldots,f_{1}]-[e_1,\ldots,e_{s_1}]$$ so that the $(-1)$-curve intersects $F_{1}$ and $E_1$, and $$\frac{\Delta}{\Omega} = [f_{s_2},\ldots,f_{1},1,e_1,\ldots,e_{s_1}]$$ where $0<\Omega<\Delta$ and $(Q \in Y)$ is
$\frac{1}{\Delta}(1,\Omega)$.

We define $\delta:= m_1 a_2 +m_2 a_1 - m_1 m_2$, and so $$\Delta= m_1^2 + m_2^2 - \delta m_1 m_2, \ \ \ \Omega= (m_2-\delta m_1)(m_2-a_2)+m_1 a_1 -1.$$ We have
$K_{X} \cdot C=\frac{- \delta}{m_1 m_2 } <0$ and $C \cdot C= \frac{- \Delta}{m_1^2 m_2^2} <0$.

\vspace{0.4cm}

\underline{($X^+ \to Y$)}: In analogy to an \enii, an extremal P-resolution has data $[f_{s_2},\ldots,f_{1}]-c-[e_1,\ldots,e_{s_1}]$, so that
$$\frac{\Delta}{\Omega}=[f_{s_2},\ldots,f_{1},c,e_1,\ldots,e_{s_1}]$$ where $-c$ is the self-intersection of the strict transform of $C^+$ in the
minimal resolution of $X^+$, $0<\Omega<\Delta$, and $(Q \in Y)$ is $\frac{1}{\Delta}(1,\Omega)$. As in an \enii~, here $\frac{{m'}_1^2}{{m'}_1
{a'}_1-1}=[e_1,\ldots,e_{s_1}]$ and $\frac{{m'}_2^2}{{m'}_2 {a'}_2-1}=[f_1,\ldots,f_{s_2}]$. If a Wahl singularity (or both) is (are) actually smooth, then we set ${m'}_i={a'}_i=1$. We define $$\delta= c{m'}_1 {m'}_2 - {m'}_1 {a'}_2 - {m'}_2 {a'}_1,$$ and so $\Delta= {m'}_1^2 + {m'}_2^2 + \delta {m'}_1 {m'}_2$ and, when both
${m'}_i \neq 1$, $$\Omega = -{m'}_1^2 (c-1) + ({m'}_2+\delta {m'}_1)({m'}_2-{a'}_2)+{m'}_1 {a'}_1 -1.$$ (One easily computes $\Omega$ when one or both ${m'}_i=1$.) We have $$K_{X^+} \cdot C^+=\frac{\delta}{{m'}_1 {m'}_2 } >0 \ \ \ \text{and} \ \ \ C^+ \cdot C^+= \frac{- \Delta}{{m'}_1^2 {m'}_2^2} <0.$$

\begin{remark}
Notice that for $0<a<m$ with $gcd(m,a)=1$, we have $\frac{m^2}{m a-1}=[e_1,\ldots,e_{s}]$ and $\frac{m^2}{m (m-a)a-1}=[e_s,\ldots,e_{1}]$, since $$(m (m-a)a-1)(ma-1) \equiv 1(\text{mod }m^2).$$ When we give the data of the Wahl singularities in an \enii~ or an extremal P-resolution, we are giving the way that the strict transform of the curve $C$ or $C^+$, respectively, intersects the exceptional divisor of the corresponding minimal resolution.
\end{remark}

We now show how to compute explicitly. First we recall Mori's algorithm to compute the numerical data of either the flip or the divisorial contraction for any \enii; cf. \cite{Mori02}.

Let us consider an arbitrary extremal neighborhood $\E$ of type $\enii$ with numerical data $(m,b)$, $(n,a)$, so that the Wahl singularities are $$\frac{1}{m^2}(1,m b -1) , \ \ \frac{1}{n^2}(1,n a -1),$$ $\delta=ma+nb-mn>0$, and $0 < \Omega < \Delta$ as above. Without loss of generality, we assume $n > m$. (Using the formulas for $\delta$ and $\Delta$, it is easy to see that $m\neq n$.) From this data, Mori constructs other extremal neighborhoods $\E'$ of type \enii ~such that both $\E$ and $\E'$ are of the same type (either flipping or divisorial), and after the birational modification the corresponding central fibers are the same. We now explain how to find these $\E'$, and Mori's criterion to know when $\E$ is flipping or divisorial.

\bigskip

Assume $\delta>1$, the case $\delta=1$ will be treated separately.

Let us define the recursion $\zeta_1=0$, $\zeta_2=1$, $$\zeta_{i+1}+\zeta_{i-1}=\delta \zeta_{i},$$ for $i \geq 2$. One can show that \begin{equation}\label{pairI}
\big(\zeta_{i+1} n-\zeta_i m,\zeta_{i+1} a-\zeta_i (m-b)\big)
\end{equation} is a pair of positive integers for all $i\geq 1$. But one can prove that there exists an integer $i_0 \geq 1$ such that \begin{equation}\label{pairII}
\big( \zeta_{i+1} m-\zeta_i n, \zeta_{i+1} b-\zeta_i (n-a) \big)
\end{equation} is a pair of positive integers only for $1 \leq i \leq i_0-1$. Precisely, we have $\zeta_{i_0+1} m-\zeta_{i_0} n \leq 0$. Two consecutive pairs of positive numbers of the form (\ref{pairI}) or (\ref{pairII}) above define the two Wahl singularities of an $\E'$, with associated numbers $\delta$, $\Omega$, and $\Delta$ (same numbers as for $\E$). Below we will show precisely the $\E'$. Mori proves that $\E$ is of flipping type if $\zeta_{i_0+1} m-\zeta_{i_0} n < 0$. Otherwise (i.e. $\zeta_{i_0+1} m-\zeta_{i_0} n = 0$) $\E$ is of divisorial type.

\bigskip

Notice that this procedure gives an initial $\E'$, right before reaching the index $i_0$. We call it the \underline{\textit{initial}} \enii~associated to a given $\E$.

\begin{example}
Let us consider an $\E$ of type \enii ~with data $(37,24)$, $(14,5)$. Here $m=14,b=5$ and $n=37,b=24$. Ones computes in this case $\Delta=11$, $\Omega=3$, and $\delta=3$. The sequence of pairs that stops is: $(14,5)$, $(5,2)$, $(1,1)$. The last $(1,1)$ means that the corresponding \enii ~is an \eni, i.e., it has only one Wahl singularity. After that, one has $3 \cdot 1 -5 <0$, and so $\E$ is of flipping type. The initial $\E'$ has one Wahl singularity $\frac{1}{25}(1,9)$. This example will continue in Example \ref{exantiflipfam}.
\label{iniexantiflipfam}
\end{example}

We now give the computation of the numerical data (as presented above) of all the $\E'$, and the corresponding flip or divisorial contraction from an initial \enii.

Consider an initial \enii ~$\E_1$ with Wahl singularities defined by pairs $(m_1,a_1)$ and $(m_2,a_2)$ with $m_2>m_1$, and numbers $\delta$, $\Delta$ and $\Omega$, where $\delta m_1 - m_2 \leq 0$. We also allow the \eni ~special case $m_1=a_1=1$.

For $i \geq 2$, we have the Mori recursions (see \cite[\S3.3]{HTU12}) $$ d(1)=m_1, \ \ \ d(2)=m_2, \ \ \ d(i-1)+d(i+1)=\delta d(i)$$ and
$c(1)= a_1$, $c(2)=m_2-a_2$, $c(i-1) + c(i+1)=\delta c(i)$ with $i \geq 3$.

When $\delta >1$, for each $i$ we have an \enii ~$\E_i$ with Wahl singularities defined by the pairs $$(m_{i},a_{i}), (m_{i+1},a_{i+1})$$ where $m_{i+1}=d(i+1), a_{i+1}=d(i+1)-c(i+1)$ and $m_i=d(i), a_i=c(i)$. We have $m_{i+1} > m_i$. The numbers $\delta$, $\Delta$ and $\Omega$, and the flipping or divisorial type of $\E_i$ are equal to the ones associated to $\E_1$. We call this sequence of \enii's a $\textit{Mori sequence}$.

If $\delta=1$, then the initial \enii ~must be flipping (by Mori's criterion), and the Mori sequence above gives only one more \enii ~with data $m_3=d(2)-d(1), a_3=d(2)-d(1) +c(1)- c(2)$ and $m_2=d(2),a_2=c(2)$.

\bigskip

From the numerical data of $\E_1$, we have according to $\delta m_1 - m_2$

\begin{itemize}
\item[\textbf{($=$0)}] $($see \cite[Prop.3.13]{HTU12}$)$ Divisorial type: then $m_1=\delta$, $m_2=\delta^2=\Delta$, $\Omega=\delta a_1-1$, and  $a_2=\delta^2-\Omega$. As in Proposition \ref{divtype}, the corresponding contraction $(X \subset \X) \to (Y \subset \Y)$ has the effect of blowing down a $(-1)$-curve $E' \subset X' \to Y'$ between smooth fibers $X'$ and $Y'$.

\item[\textbf{($<$0)}] $($see \cite[Prop.3.15, Thm.3.20]{HTU12}$)$ Flipping type: the extremal P-resolution $X^+$ has ${m'}_2=m_1$, ${a'}_2=m_1-a_1$, and ${m'}_1=m_2-\delta m_1, \ {a'}_1 \equiv (m_2-a_2)-\delta a_1 \ (\text{mod} \ {m'}_1)$. If $m_1=a_1=1$, then we set ${a'}_2=1$. The self-intersection of $C^+$ can be found using the formula for $\delta$ for an extremal P-resolution, see the numerical description above.
\end{itemize}

\begin{remark}
For a given Wahl singularity $\frac{1}{\delta^2}(1,\delta a-1)$ we have one Mori sequence of divisorial type starting with the data in \textbf{($=$0)}. For a given extremal P-resolution $X^+$, we have at most two corresponding Mori sequences, one for each of the Wahl singularities in $X^+$. This is in \cite[Cor.3.23]{HTU12}, and the precise procedure can be read from either above or from the last part of the proof of \cite[Cor.3.23]{HTU12}. We do not give details here because we will not use it.
\label{anti}
\end{remark}

In \cite{HTU12} we show how to compute for all extremal neighborhoods of type \eni. More precisely, we prove that a given exceptional neighborhood of type \eni ~degenerates to two \enii ~sharing the type, and the central fiber of the resulting birational operation.

\begin{proposition} \cite[\S2.3 and \S3.4]{HTU12}
Let $[e_1,\ldots,\overline{e_i},\ldots,e_s]$ be the data of an \eni ~with $\frac{m^2}{ma-1}=[e_1,\ldots,e_s]$. Let $\delta, \Delta,
\Omega$ be as in the above numerical description of an \eni. Let $\frac{m_2}{m_2-a_2}=[e_1,\ldots,e_{i-1}]$ and $\frac{m_1}{m_1-a_1}=[e_s,\ldots, e_{i+1}]$, if possible (this is, for the first $i>1$, for the second $i<s$). Then, there are \enii ~with data $$[f_{s_2},\ldots,f_{1}]-[e_1,\ldots,e_s] \ \ \
\text{and} \ \ \ [e_1,\ldots,e_s]-[g_1,\ldots,g_{s_1}],$$ where $\frac{m_2^2}{m_2a_2-1}= [f_1,\ldots,f_{s_2}]$,
$\frac{m_1^2}{m_1a_1-1}= [g_1,\ldots,g_{s_1}]$, such that the corresponding cyclic quotient singularity
$\frac{1}{\Delta}(1,\Omega)$ and $\delta$ are the same for the \eni ~and the \enii. Moreover, each of the \enii ~deforms (over a smooth analytic germ of a curve) to the \eni ~by $\Q$-Gorenstein smoothing up $\frac{1}{m_i^2}(1,m_ia_i-1)$ while keeping $\frac{1}{m^2}(1,ma-1)$, and there are two possibilities: either these three extremal neighborhoods are

\begin{enumerate}
\item flipping, with the same extremal P-resolution for the flip, or

\item divisorial, with the same $(Q \in Y)$.
\end{enumerate}
\label{InumII}
\end{proposition}

Therefore Proposition \ref{InumII} allows us to compute the flip or the divisorial contraction for any \eni ~through the Mori algorithm \cite{Mori02} for extremal neighborhoods of type k2A described above. In \cite{HTU12} we show that this gives a complete description of the situation, which provides a universal family for both flipping and divisorial contractions; see \cite[\S3]{HTU12}. Below we show a complete example in each case.

\begin{example} (Divisorial family) Consider the Wahl singularity $(Q \in Y)=\frac{1}{4}(1,1)$. So $\Delta=4$ and $\Omega=1$, and $\delta=2$. Then the numerical data of any \eni ~and any \enii ~of divisorial type associated to $(Q \in Y)$ can be read from $$[4]-[2,\bar{2},6]-[2,2,2,\bar{2},8]-[2,2,2,2,2,\bar{2},10]-\cdots $$ Notice that $\delta=2$. For example, $[2,2,2,2,2,\bar{2},10]$ is an \eni~, and $[2,\bar{2},6]-[2,2,2,\bar{2},8]$ is an \enii.
\label{exdivfam}
\end{example}

\begin{example} (Flipping family) Let $\frac{1}{11}(1,3)$ be the cyclic quotient singularity $(Q \in Y)$. So $\Delta=11$ and $\Omega=3$. Consider the extremal P-resolution $X^+ \to Y$ defined by $[4]-3$. Here ${m'}_1={a'}_1=1$, ${m'}_2=2$, ${a'}_2=1$, $\delta=3$, and the ``middle" curve is a $(-3)$-curve. Then the numerical data of any \eni ~and any \enii ~associated to $X^+$ can be read from
$$[\bar{2},5,3]-[2,3,\bar{2},2,7,3]-[2,3,2,2,2,\bar{2},5,7,3]-\cdots
$$ and
$$[4]-[2,\bar{2},5,4]-[2,2,3,\bar{2},2,7,4]-[2,2,3,2,2,2,\bar{2},5,7,4]-\cdots$$ These two Mori sequences provide the numerical data of the universal antiflip \cite[\S3]{HTU12} of $[4]-3$. For particular examples, we have that $[2,3,\bar{2},2,7,3]$ and $[2,\bar{2},5,4]$ are \eni~ whose flips have $X^+$ as central fiber.
\label{exantiflipfam}
\end{example}

A flip which appears frequently in calculations is the following

\begin{proposition}
Let $[e_1,\ldots,e_{s-1},\overline{e_s}]$ be a flipping \eni. Let $i \in \{1,\ldots,s\}$ be such that $e_i \geq 3$ and $e_j=2$ for all $j>i$. (If $e_s>2$, then we set $i=s$.)

Then the data for $X^+$ is $e_1-[e_2,\ldots,e_i-1]$.
\label{specialFlip}
\end{proposition}

\begin{proof}
Write $\frac{m^2}{ma-1}=[e_1,\ldots,e_s]$. Notice that according to our numeric description for \eni ~we have $\beta_s=1$, $\alpha_s=m(m-a)-1$, and $\gamma_s=a(m-a)-1$. Therefore $\delta=n-a$, $\Delta=na+1$, and $\Omega=a^2$, following the formulas above. Notice also that in this case the \eni ~we are considering can be seen as an initial \enii ~by taking $m_2=n$, $a_2=n-a$, $m_1=1$, and $a_1=1$. One can recompute that $\delta=n-a$, $\Delta=na+1$, and $\Omega=a^2$ following the formulas above, and that $\delta m_1-m_2=-a <0$, and so it is indeed of flipping type. To compute the numerical data of $X^+$, we use the formulas in \textbf{($<$0)} above: $m'_2=1$, $a'_2=1$, $m'_1=a$, and $0<a'_1<a$ such that $a'_1 \equiv -n ($mod $a)$.

Notice that if $a=1$, then we have our claim. So we assume that $a>1$. Then, by definition, $\frac{\Delta}{\Omega}=[e_1,\ldots,e_i-1]$, and so $\frac{na+1}{a^2}=[e_1,\ldots,e_i-1]$. This gives $$\frac{a^2}{a(ae_1-n)-1}=[e_2,\ldots,e_i-1].$$ But $e_1$ is the integral part of $\frac{n^2}{na-1}$ plus $1$, and so $0<ae_1-n<a$. Therefore, when $a>1$, we have precisely $a'_1=ae_1-n$, and our claim follows.
\end{proof}

A corollary is the useful fact (to be used in the next sections)

\begin{proposition}
\cite[p.188]{HP2010} Let $\widetilde{Y}$ be a smooth surface with a chain of rational smooth curves $E_1,\ldots,E_s$, which is the exceptional divisor of a Wahl singularity. Let $C_1, C_2$ be $(-1)$-curves in $\widetilde{Y}$ such that $C_1 \cdot C_2 =0$, $C_1 \cdot E_1=1$, and $C_2 \cdot E_s=1$, and $C_1, C_2$ do not intersect any other $E_i$'s. Let $\sigma \colon \widetilde{Y} \to Y$ be the contraction of the chain $E_1,\ldots,E_s$ (to a Wahl singularity), and let $C_0=\sigma(C_1) \cup \sigma(C_2)$. Assume there is a $\Q$-Gorenstein smoothing $(Y \subset \Y) \to (0 \in \D)$.

Then there is a $(-1)$-curve $C_t$ in the smooth fiber over $t \in \D \setminus \{0\}$ which degenerates to $C_0$.
\label{(-1)}
\end{proposition}

\begin{proof}
Notice that $C:=\sigma(C_2)$ defines an \eni ~of flipping type as in Proposition \ref{specialFlip}. After we perform the flip, we obtain a surface $Y^+$ (from the corresponding extremal P-resolution) and the strict transform of $\sigma(C_1)$ in $Y^+$ does not pass through the singularity. Therefore the $\Q$-Gorenstein smoothing of $Y^+$, which gives the flip, would have a $(-1)$-curve $C_t$ in the general fiber that deforms to $\sigma(C_2)$. It is clear that in $(Y \subset \Y) \to (0 \in \D)$ this $(-1)$-curve degenerates to $C_0$.
\end{proof}

\section{Method of identification} \label{identify}

We now explain the method to identify stable surfaces around the stable model of a given Lee-Park surface, i.e., a normal projective surface with only Wahl singularities, and no local-to-global obstructions (any local deformation of its singularities may be globalized).

\subsection{Stable model of a Lee-Park surface} \label{method1} Let $W$ be a normal projective surface with only Wahl singularities, and $H^2(W,T_W)=0$. Then $W$ has no local-to-global obstructions; see \cite[\S2]{LP07}. We remark that the vanishing of $H^2(W,T_W)$ is commonly achieved by the vanishing of  $H^2(\widetilde{W},T_{\widetilde{W}}(- \log E))$, where $E$ is the exceptional divisor of the minimal resolution $\widetilde{W} \to W$; see \cite[Thm.2]{LP07}. Assume that $K_W$ is nef, and that $K_W^2 >0$. Let $(W \subset \W) \to (0 \in \D)$ be a $\Q$-Gorenstein smoothing of $W$ (Definition \ref{def}). Then we know that the general fiber $W'$ has $K_{W'}$ nef (see \cite[p.499]{LP07}), and $K_{W}^2=K_{W'}^2>0$. Thus $W'$ is a minimal surface of general type.

The canonical class $K_W$ may not be ample. To find the stable model $\overline{W}$ of $W$, one considers the relative canonical model of $(W \subset \W) \to (0 \in \D)$. The following lemma tells us what type of singularities we can expect in $\overline{W}$.

\begin{lemma}
The relative canonical model $(\overline{W} \subset \overline{\W}) \to (0 \in \D)$ of the $(W \subset \W) \to (0 \in \D)$ above has as central fiber a normal projective surface $\overline{W}$ with only T-singularities (Definition \ref{T-sing}). \label{l2}
\end{lemma}

\begin{proof}
We know there is $(\overline{W} \subset \overline{\W}) \to (0 \in \D)$; cf. \cite{KM1998}. We have a birational morphism $\W \to \overline{\W}$ over $\D$ such that $K_{\overline{\W}}$ is $\Q$-Cartier and ample. Notice that $\overline{W}$ has log terminal singularities because $W$ does \cite[pp.102--103]{KM1998}. The singularities of $\overline{W}$ must be T-singularities by \cite[\S5.2]{KSB88}.
\end{proof}

Thus $\overline{W}$ can have only du Val singularities, and cyclic quotient singularities $\frac{1}{dn^2}(1,dna-1)$ with gcd$(n,a)=1$. In addition, locally around each singularity of $\overline{W}$, we have that $W \to \overline{W}$ is a \emph{M-resolution}; see \cite{BC94}. We will use that interpretation below.

The surface $\overline{W}$ is a point in the KSBA compactification of the moduli space of surfaces of general type $\overline{\M}_{K_W^2,\chi(\O_W)}$ with fixed topological invariants $K_W^2$, $\chi(\O_W)$; cf. \cite{H2011}. We know that $\overline{\M}_{K_W^2,\chi(\O_W)}$ at $\overline{W}$ is locally a finite quotient of the smooth germ $\Def^{\QG}(\overline{W})$ of dimension $10 \chi(\O_W)- 2 K_W^2 $, where $\Def^{\QG}(\overline{W})$ is the versal $\Q$-Gorenstein deformation space of $\overline{W}$. The smoothness follows from $H^2(\overline{W},T_{\overline{W}})=0$ (which follows from $H^2(W,T_W)=0$); see \cite[Sect.3]{H2011}. The local dimension is a Riemann-Roch calculation: see the proof of \cite[Prop.2.2]{PSU} for example.

The following lemma will be used to identify $\overline{W}$ in sections \S \ref{s1}, \S \ref{s2}, and \S \ref{s3}.

\begin{lemma}
Let $Z \to \P^1$ be an elliptic fibration, where $Z$ is a rational smooth projective surface. Assume it has two fibers $F_1$, $F_2$ of type $I_1$, and two sections $P,Q$. Let $Z''$ be the surface obtained by blowing up the nodes of both $F_1$ and $F_2$ in $Z$, and blowing down $P$ and $Q$. Then $Z''$ is a Halphen surface \cite[\S2]{CD12} of index $2$, i.e., $Z''$ has an elliptic fibration with a unique multiple fiber of multiplicity $2$. The curve $F_1+F_2$ in $Z''$ is a non-multiple fiber of type
$I_2$.
\label{l1}
\end{lemma}

\begin{proof}
Let $\pi \colon Z \to \P^2$ be a blow-down to $\P^2$ starting with the sections $P,Q$ (see proof of \cite[Prop.2.2]{CD12} for
example). Then, the elliptic fibration $Z \to \P^1$ comes from the pencil of cubics $$\{ af_1+bf_2 \colon (a:b) \in \P^1 \},$$ where
$f_1$, $f_2$ are the cubic polynomials of the images of $F_1$, $F_2$ under $\pi$. Notice that the node of $F_i$ is not in $F_j$
for $i\neq j$. Hence there exists a unique cubic $\Lambda$ passing through the node of $F_1$, the node of $F_2$, and the $7$ base
points of the pencil above not including the ones corresponding to $P$ and $Q$. This gives the existence of the Halphen pencil of
index $2$ $$\{ c f_1 f_2 + d \lambda^2 \colon (c:d) \in \P^1 \}$$ where $\lambda=0$ is the
equation of $\Lambda$. The associated Halphen surface is the $Z''$ described in the statement of this lemma.
\end{proof}

\subsection{Partial $\Q$-Gorenstein smoothings} \label{method2}

Each of the non du Val T-singularities $\frac{1}{dn^2}(1,dna-1)$ of $\overline{W}$ defines a divisor $\DD {n \choose a}$ in $\overline{\M}_{K_W^2,\chi(\O_W)}$. A general point in this divisor represents a normal KSBA surface with one Wahl singularity $\frac{1}{n^2}(1,na-1)$. Our main goal is to identify as much as possible the smooth minimal model of that surface.

\begin{remark}
Du Val singularities have simultaneous resolutions in deformations. Thus we know that a $\Q$-Gorenstein smoothing of all the non du Val T-singularities of $\overline{W}$ has as general fiber the canonical model of a smooth projective surface of general type with invariants $K_W^2$ and $\chi(\O_W)$. There is no identification problem in this case.
\end{remark}

The divisor $\DD {n \choose a}$ is defined in the following way. We have $\overline{W}$ with no local-to-global obstructions. We consider a $\Q$-Gorenstein deformation of $\overline{W}$ which locally deforms a given T-singularity $\frac{1}{dn^2}(1,dna-1)$ into $\frac{1}{n^2}(1,na-1)$ (see \cite[\S2.1]{BC94} or \cite[Prop.2.3]{HP2010}), and smooths up all other singularities of $\overline{W}$. The general fiber of this deformation is a KSBA surface with one Wahl singularity. This surface defines the divisor $\DD {n \choose a}$; see \cite[\S4]{H2011}. To identify it, we will run MMP. But we will use another suitable family to run it, because we want to use only birational operations to type \eni~ and \enii; cf. \cite[\S5]{HTU12}. We explain that below.

Locally at each T-singularity, the birational map $W \to \overline{W}$ is an M-resolution. In particular, du Val singularities are resolved, and over a singularity of type $\frac{1}{dn^2}(1,dna-1)$ we have $d$ Wahl singularities of type $\frac{1}{n^2}(1,na-1)$.

\begin{lemma}
Any $\Q$-Gorenstein deformation of $\overline{W}$ is induced by blowing down a $\Q$-Gorenstein deformation of $W$.
\end{lemma}

\begin{proof}
See proof of \cite[Lemma 5.2]{HTU12}. We use the local picture of $M$-resolutions, and the blowing-down deformation result of \cite{BC94}.
\end{proof}

Therefore, to identify the general surface in $\DD {n \choose a}$ we consider a $\Q$-Gorenstein deformation of $W$ which is locally trivial on one of the Wahl singularities $\frac{1}{n^2}(1,na-1)$ above the given T-singularity $\frac{1}{dn^2}(1,dna-1)$, and smooths up all other singularities in $W$. It does not matter which Wahl singularity we choose over $\frac{1}{dn^2}(1,dna-1)$, we will always land in the same divisor $\DD {n \choose a}$. This is because locally the blowing-down deformation includes a transitive action on the $d$ Wahl singularities; see \cite[\S2]{BC94}.

Let $X_0 \to W$ be the resolution of the chosen Wahl singularity $\frac{1}{n^2}(1,na-1)$. Since the above $\Q$-Gorenstein deformation of $W$ is trivial around this Wahl singularity, we can and do resolve it simultaneously. With this, we obtain a $\Q$-Gorenstein smoothing $(X_0 \subset \X_{0}) \to (0 \in \D)$ such that the general fiber $X'_{0}$ is the minimal resolution of the surface we want to identify (and so contains the exceptional divisor of $\frac{1}{n^2}(1,na-1)$).
\subsection{Running MMP explicitly} \label{method3}

If $K_{X_0}$ is not nef, then we run the explicit MMP in \cite[\S5]{HTU12} on the extremal neighborhood defined by $$(X_0 \subset \X_{0}) \to (0 \in \D).$$ Our purpose is to find the relative minimal model of $\X_{0} \to \D$. The general fiber of the minimal model will be the minimal model of $X'_0$.

There will be several flips and divisorial contractions over $\X_{0}$, all of them of type \eni ~or \enii; cf. \cite[Thm.5.3]{HTU12}. For each birational operation, we denote the corresponding $\Q$-Gorenstein smoothing by $(X_i \subset \X_{i}) \to (0 \in \D)$, whose general fiber is $X'_i$.

After certain finite $n$ steps, two situations may arise: we have that $(X_n \subset \X_{n}) \to (0 \in \D)$ has either $K_{X_n}$ nef, or the surface $X_n$ is smooth. In the latter, we have a smooth deformation, and so the Kodaira dimension of $X_n$ and $X'_n$ coincide. In this case we will be able to identify the minimal model of $X'_n$, since any possible $(-1)$-curve in $X_n$ lifts to a $(-1)$-curve in $X'_n$; cf. \cite[IV \S4]{BHPV04}. If, on the other hand, we have $X_n$ singular but $K_{X_n}$ nef, then the general fiber is the minimal model we wanted to find.

We construct $X_i$ from $X_{i-1}$ via the following procedure: if we are not in one of the above situations, then in $X_{i-1}$ there is a a smooth rational curve $C_{i} \subset X_{i-1}$ such that $C_{i} \cdot K_{X_{i-1}} <0$ and $C_{i} \cdot C_{i} < 0$, which is as in \S\ref{enbhd}. Hence $C_{i}$ becomes a $(-1)$-curve in the minimal resolution of the Wahl singularities it contains. After we perform the birational operation, we have two possibilities for the new $(X_i \subset \X_{i}) \to (0 \in \D)$: it is the result of either a divisorial contraction, and so between general fibers we have the blow-down of a $(-1)$-curve $X'_{i-1} \to X'_{i}$ (Proposition \ref{divtype}), or a flip, so that the general fibers $X'_{i-1}$, $X'_{i}$ are isomorphic. We find $X_{i}$ as the $X_{i-1}^+$ of the flip, see \S\ref{enbhd}.

Notice that in both cases the surface $X_i$ is birational to $X_{i-1}$. The operations roughly are: minimally resolve $X_{i-1}$ at the Wahl singularities in $C_i$, then contract the strict transform of $C_i$ and all other $(-1)$-curves coming from the exceptional divisor, then perform certain other blow-ups required to find the corresponding extremal P-resolution (this is not required in case of divisorial contraction), and finally contract the configurations corresponding to the Wahl singularities we need for $X^+=X_i$. These birational operations modify curves only over $C_i$. In particular the transformations on $X_{i-1}$ do not affect singularities outside of $C_i$.

Also, since the amount of information is big, we will codify all birational operations in dots diagrams, which are explained in detail in \cite[Notation 5.5]{HTU12}. They basically show the transformation of relevant curves under flips and divisorial contractions in the minimal resolution of $X_i$. In the next sections we will do this using Lee-Park surfaces.

One may wonder at this point what sort of surfaces with only Wahl singularities one can expect in the KSBA boundary. The following proposition, due to Kawamata \cite{K92}, says that at least there is a hierarchy with respect to $K^2$ and the Kodaira dimension.

\begin{proposition}
Let $(W \subset \W) \to (0 \in \D)$ be a $\Q$-Gorenstein smoothing of a normal singular projective surface $W$ with only Wahl singularities. Let $\widetilde{W}$ be the minimal resolution of $W$, and let $Z$ be the smooth minimal model of $\widetilde{W}$. Assume that $K_{\W}$ is relatively nef. If $Z$ is of general type, then the general fiber $W'$ is of general type and $K_{W'}^2=K_{W}^2 > K_{Z}^2$. \label{p5}
\end{proposition}

\begin{proof}
By Kawamata \cite[Lemma 2.4]{K92}, there exist positive integers $m_1$ and $m_2$ such that the inequalities of $m$-plurigenera $P_m(W') > P_m(Z)$ hold for positive integers $m$ with $m_1$ dividing $m$ and $m_2 < m$. This implies that $W'$ is of general type. Moreover, this inequality becomes \cite[VII Cor(5.4)]{BHPV04}
$\frac{m(m-1)}{2} K_{W'}^2 + \chi(W') > \frac{m(m-1)}{2} K_{Z}^2 + \chi(Z)$ for those $m$, and so we have the claim.
\end{proof}

\begin{remark}
We will show examples where the final family $(X_n \subset \X_{n}) \to (0 \in \D)$ has $X_n$ singular with $K_{X_n}$ nef, and the general fiber $X'_{n}$ is of general type (see Prop. \ref{p6} for the simplest). In this way, the resulting surface $X_n$ represents, after going to KSBA model (Lemma \ref{l2}), a stable surface in the KSBA moduli space which contains $X'_{n}$. Notice that by Prop. \ref{p5}, this KSBA moduli space is not the one we started with.
\label{newfam}
\end{remark}

\section{Elliptic surfaces via $\Q$-Gorenstein smoothings} \label{app1}

This is a complementary section which will be used in the identification of some KSBA surfaces in \S\ref{s1}, \S\ref{s2}, and \S\ref{s3}.

The exceptional divisor of any T-singularity $\frac{1}{d n^2}(1,dna-1)$ can be obtained from an $I_d$ elliptic singular fiber by blowing up over a node. We blow up a node of $I_d$ and subsequent nodes coming from the new $(-1)$-curves. The exceptional divisor appears as the chain of curves of the total
transform of $I_d$ which does not contain the (last) $(-1)$-curve (see \cite[Prop.3.11]{KSB88}). We call this construction a \textit{T-blow-up} of $I_d$. This way of looking at T-singularities is essentially in Kawamata's paper \cite{K92} \footnote{He writes $\frac{1}{r^2}(a,r-a)$ instead of $\frac{1}{r^2}(1,ra^{-1}-1)$, where $0<a^{-1}<r$ and $aa^{-1} \equiv 1($mod $r)$.}.

If $g \colon Z \to B$ is an elliptic fibration over a smooth curve $B$ with a singular fiber $I_d$, then we denote by $\sigma \colon \widetilde{Z} \to Z$ the composition of blow-ups used in a T-blow-up of $I_d$. Let $\{E_1, \ldots, E_s \}$ be the corresponding T-configuration where $\frac{1}{dn^2}(1,dna-1)=[e_1,\ldots,e_s]$, and $E_i^2=-e_i$. Write $\sigma^*(I_d)= \sum_{i=1}^{s+1} \nu_i E_i$, where $E_{s+1}$ is the $(-1)$-curve, and $\nu_i \geq 1$ are integers.

\begin{lemma}
In a situation as above, we have $n=\nu_{s+1}$, $a=\nu_{s+1}-\nu_s$, and the discrepancy of $E_i$ is $-1 + \frac{\nu_i}{\nu_{s+1}}$ for all $i=1,\ldots,s$. \label{numerics}
\end{lemma}

\begin{proof}
The proof is based in \cite[Lemma 3.4]{Stevens1989} and induction on the number of blow-ups. If we have only one blow-up, i.e. the T-singularity is either $[4]$ or $[3,2,...,2,3]$, then the discrepancies are all $-\frac{1}{2}$, which agrees with our claim. Hence, using the hypothesis of induction for a length $s$ T-singularity, one can easily see by \cite[Lemma 3.4]{Stevens1989} that for a length $s+1$ T-singularity we have our claim.
\end{proof}

\begin{theorem}
Let $g \colon Z \rightarrow \P^1$ be a relatively minimal elliptic fibration with a section, such that $Z$ is a rational smooth projective surface.

\vspace{0.1cm}

$(-\infty):$ Assume $g$ has a fiber of type $I_d$. Consider a T-blow-up of $I_d$ with the notation above. Let $\widetilde{Z} \to W$ be the contraction of the T-configuration. Then there are $\Q$-Gorenstein smoothings $W'$ of $W$, and any such $W'$ is rational.

\vspace{0.1cm}

$(0):$ Assume $g$ has two fibers $I_{d_1}$ and $I_{d_2}$.
Let $\widetilde{Z}$ be the blow-up of $Z$ at one node of $I_{d_1}$ and at one
node of $I_{d_2}$. Hence we have two T-configurations of type
$\frac{1}{4d_i}(1,2d_i-1)$. Let $\widetilde{Z} \to W$ be the contraction of these
configurations. Then there are $\Q$-Gorenstein smoothings $W'$ of
$W$, and any such $W'$ is an Enriques surface.

\vspace{0.1 cm}

$(1):$ Assume it has two fibers $I_{d_1}$ and $I_{d_2}$. We
apply T-blow-ups to each of them. Assume that for one of them we
blew-up at least twice. Let $\widetilde{Z} \to W$ be the contraction of both
T-configurations. Then there are $\Q$-Gorenstein smoothings
$W'$ of $W$, and any such $W'$ has Kodaira dimension $1$.

\label{t0}
\end{theorem}

\begin{proof}
For the proof, we assume $g$ has the singular fibers $I_{d_1}$ and
$I_{d_2}$. This situation adjusts to prove all cases
simultaneously. Let $\sigma \colon \widetilde{Z} \to Z$ be the
composition of blow ups for both T-blow-ups, so that $\widetilde{Z}$ contains
the T-configurations $\{E_1, \ldots, E_s \}$ and $\{F_1, \ldots,
F_r \}$ of types $\frac{1}{d_1 n_1^2}(1,d_1 n_1 a_1
-1)=[e_1,\ldots,e_s]$ and $\frac{1}{d_2 n_2^2}(1,d_2 n_2 a_2
-1)=[f_1,\ldots,f_r]$, where $E_i^2=-e_i$ and $F_i^2=-f_i$. We
also have the $(-1)$-curves $E_{s+1}$ and $F_{r+1}$, so that
$\sigma^*(I_{d_1})= \sum_{i=1}^{s+1} \nu_i E_i$, and
$\sigma^*(I_{d_2})=\sum_{i=1}^{r+1} \mu_i F_i $. Let $h \colon \widetilde{Z}
\to W$ be the contraction of both T-configurations.

Through arguments as in \cite{LP07} (see \cite[\S4]{PSU}), we know that
$$H^2(\widetilde{Z}, T_{\widetilde{Z}}(- \log ( E_1 + \ldots + E_s + F_1 + \ldots + F_r  )))=0,$$ and so there are
no local-to-global obstructions to deform $W$.

Let $C$ be the general fiber of $g$. Then, $$ K_{\widetilde{Z}} \sim
-\sigma^{*} C + \sum_{i=1}^{s+1} (\nu_i-1)E_i + \sum_{i=1}^{r+1}
(\mu_i-1)F_i $$ and $K_{\widetilde{Z}} \equiv h^* K_{W} - \sum_{i=1}^{s}
\disc(E_i) E_i - \sum_{i=1}^{r} \disc(F_i) F_i$, where $\disc$
stands for minus the discrepancy. Then, we know by Lemma \ref{numerics} that $\disc(E_i)=1 - \frac{\nu_i}{n_1}$ and $\disc(F_i)=1 -
\frac{\mu_i}{n_2}$. In this way, we have $$ h^*(K_{W}) \equiv -
\frac{1}{n_1}  \sum_{i=1}^{s+1} \nu_i E_i \equiv - \frac{1}{n_1}
\sigma^{*} C$$ for the case \textbf{(-$\infty$)}, and $$ h^*(K_{W})
\equiv \frac{n_1-2}{2n_1} \sum_{i=1}^{s+1} \nu_i E_i +
\frac{n_2-2}{2n_2} \sum_{i=1}^{r+1} \mu_i F_i \equiv \Big(1 -
\frac{1}{n_1} - \frac{1}{n_2} \Big) \sigma^{*} C$$ for cases
$(0)$ and $(1)$.

Since we are $\Q$-Gorenstein smoothing up T-singularities over $\D$, we have that $\pm K_{W}$ nef implies $\pm K_{W'}$ nef, and $K_{W} \equiv 0$ implies $K_{W'} \equiv 0$. Then, in case $(-\infty)$ we have that $-K_{W}$ is nef and not $\equiv 0$, and so $W'$ is a rational surface. We recall that in any case, $K_{W'}^2=0$, $q(X)=p_g(X)=0$. (See \cite{GS83} for the irregularity, which is constant in families, and then $p_g(W')$ follows.) For the case $(0)$ we see that $K_{W} \equiv 0$ and so for $K_{W'}$. It follows that $W'$ is an Enriques surface. For the last case $(1)$, $K_{W}$ is nef and not trivial, and so $W'$ is a minimal surface with Kodaira dimension $1$.

\end{proof}

We recall that a \emph{Dolgachev surface} of type $n_1,n_2$ is a
simply connected elliptic fibration with exactly two multiple
fibers of multiplicities $n_1$ and $n_2$; cf.
\cite[p.383]{BHPV04}.

\begin{corollary}
If in case $(1)$ we have gcd$(n_1,n_2)=1$, then a smooth
fiber of any $\Q$-Gorenstein smoothing is a Dolgachev surface of
type $n_1,n_2$. \label{dolga}
\end{corollary}

\begin{proof}
In \S\ref{cyclic}, we define sequences of integers $\{\alpha_j\}_{j=1}^s, \{\beta_j\}_{j=1}^s, \{\gamma_j\}_{j=1}^s$ for any Hirzebruch-Jung continued fraction $\frac{m}{q} =[b_1,\ldots,b_s]$. In particular, we saw that the discrepancy of the $E_j$ exceptional curve is $-1 + \frac{\alpha_j + \beta_j}{m}$. We now give some facts from \cite{Mumford61}. The fundamental group of a neighborhood of the complement of the exceptional divisor $\bigcup_{i=1}^s E_i$ is cyclic of order $m$, and it is generated by a loop $\xi$ around $E_1$ (or $E_s$). For any $j$, a loop $\xi_j$ around $E_j$ is a conjugate to $\xi^{\alpha_j}$ (or $\xi^{\beta_j}$) \cite[p.20]{Mumford61}. We now specialize to the case of T-singularities. For $m=dn^2$ and $q=dna-1$ with gcd$(n,a)=1$, we have $\beta_{j}+\alpha_{j}=\nu_{j} n$ by Lemma \ref{numerics}. On the other hand, in \S \ref{cyclic} we have the formula $\beta_{j}=(dna-1) \alpha_{j} - dn^2 \gamma_{j}$, and so \begin{equation} \label{poto} \nu_{j}=a \alpha_{j}- n \gamma_{j}. \end{equation}

Following the strategy in \cite[p.493]{LP07}, we now compute the fundamental group of a smooth fiber of a $\Q$-Gorenstein smoothing. The computation is done on the minimal resolution $\widetilde{W} \to W$ of the singular fiber $W$. It is enough to show that $\pi_1(\widetilde{W} \setminus E)$ is
trivial, where $E$ is the exceptional divisor. We consider two small loops $\xi$ and $\rho$ around the two components of $E$ which intersect a given section (we do have sections) of the elliptic fibration. We notice that for those components, the multiplicities $\nu_{j(i)}$ ($i=1,2$) are both equal to $1$. Then, by the equation (\ref{poto}), we obtain gcd$(\beta_{j(i)},n_i)=1$ for $i=1,2$. In this way, by the facts in the previous paragraph, these loops generate the fundamental groups of the neighborhoods of the complements of each component of $E$. The chosen section, which is a $\P^1$, gives that $\xi$ is conjugated to $\rho$. We now use that gcd$(n_1,n_2)=1$ to conclude that $\xi$ and $\rho$ become trivial in $\pi_1(\widetilde{W} \setminus E)$. This implies that $\pi_1(\widetilde{W} \setminus E)=1$.

Therefore, the smooth fiber $W'$ is a simply connected elliptic fibration with exactly two coprime multiple fibers; cf. \cite[II \S3]{D77}. The Kodaira dimension of $W'$ is $1$. By \cite[Thm.4.2]{K92}, the elliptic fibration $W' \to \P^1$ degenerates to the elliptic fibration $W \to \P^1$, so that the general fiber $F'$ of $W' \to \P^1$ deforms to the general fiber $F$ of $W \to \P^1$. Since this is a $\Q$-Gorenstein smoothing, we know that there exists $m$ so that the line bundle $m K_{W'}$ deforms to the line bundle $mK_W$. Let $n'_1, n'_2$ be the coprime multiplicities of $W' \to \P^1$. Then by the canonical formula $$n'_1 n'_2 K_{W'} \sim (n'_1 n'_2 -n'_1 -n'_2) F',$$ and so, by choosing $m=n'_1n'_2k$ for some suitable $k$, we have $n'_1 n'_2k K_{W} \sim (n'_1 n'_2 -n'_1 -n'_2)k F$ in $W$. But on $W$ we also have a canonical formula (see \cite[Thm.4.4]{K92}) which numerically gives $n_1 n_2 K_W \equiv (n_1 n_2 -n_1 - n_2)F$, and this implies $\big( n'_1 n'_2 (n_1 n_2 -n_1 -n_2)- n_1 n_2 (n'_1 n'_2 -n'_1 -n'_2) \big) F \equiv 0,$ and so $n'_1 n'_2 (n_1 n_2 -n_1 -n_2)=n_1 n_2 (n'_1 n'_2 -n'_1 -n'_2)$. But the pairs $(n_1,n_2)$ and $(n'_1,n'_2)$ are coprime. Then, up to permutation, they must be equal.
\end{proof}

\section{$K^2=1$} \label{s1}

We begin with the example corresponding to Figure 5 in \cite{LP07}. Consider the pencil of cubics in $\P_{x_0,x_1,x_2}^2$ $$\alpha
x_0^3 + \beta x_1(x_0^2+x_1^2-x_2^2) =0 $$ with $(\alpha : \beta) \in \P_{\alpha,\beta}^1$. We have base points $p=(0 : 1 : 1)$, $q=(0 : 1 : -1)$, and
$r=(0 : 0 : 1)$. We blow up three times each of them, to obtain an elliptic fibration $g \colon Z \to \P^1$ with a configuration of
singular fibers $IV^*,2I_1,I_2$. Let $A=\{x_0=0\}$, $B=\{x_1=0\}$, and $C=\{x_0^2+x_1^2=x_2^2\}$. Let $P$ and $Q$ be the last
exceptional divisors over $p$ and $q$. More notation is shown in Figure \ref{f1}.

\begin{figure}[htbp]
\includegraphics[width=8cm]{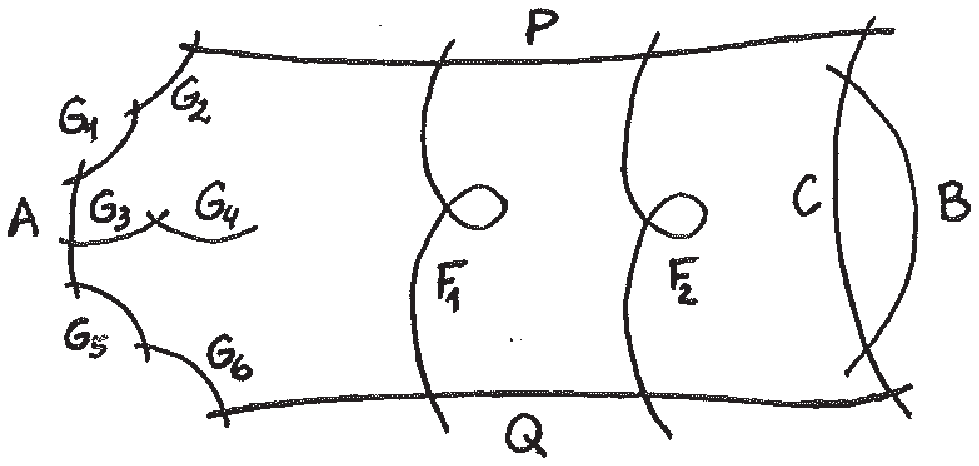}
\caption{Elliptic fibration with $IV^*,2I_1,I_2$}  \label{f1}
\end{figure}

We now blow up $Z$ $11$ times as in Figure $5$ of \cite{LP07} (see Figure \ref{f2}). Let $\widetilde{Z}'$ be the corresponding surface, and let
$X$ be the singular normal projective surface obtained by contracting the configurations of curves $[2,2,2,7]$, $[4]$, $[6,2,2]$, and $[2,6,2,3]$.

\begin{figure}[htbp]
\includegraphics[width=8cm]{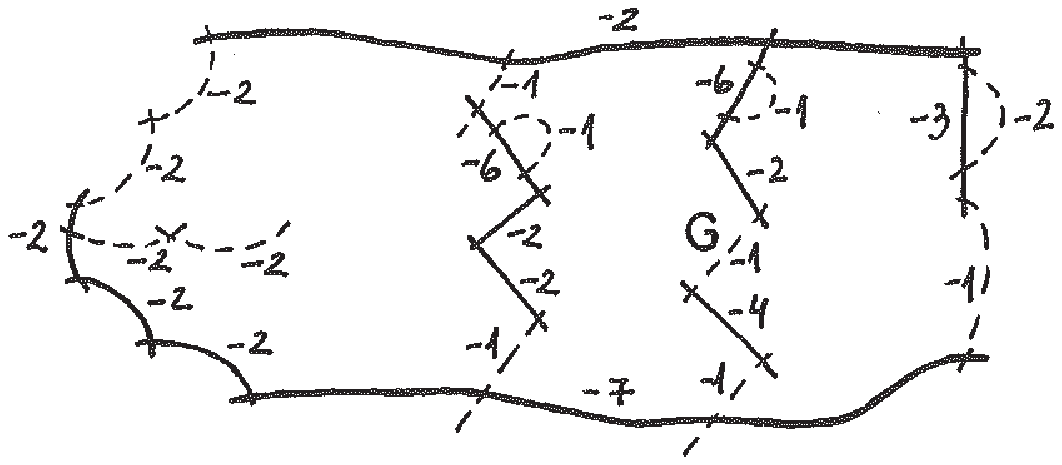}
\caption{The blow-up $\widetilde{Z}'$ of $Z$ $11$ times}  \label{f2}
\end{figure}

We have that $K_{X}$ is not nef: the intersection of the image of $G$ (see Figure \ref{f2}) in $X$ with $K_{X}$ is $-1+ \frac{3}{7} + \frac{1}{2} = -\frac{1}{14}$. However, a $\Q$-Gorenstein smoothing of these $4$ singularities indeed has the properties claimed in \cite{LP07}. To see this, we perform a flip of type \enii ~(Definition \ref{EN}) on a $\Q$-Gorenstein smoothing of $X$ over $\D$. We are flipping the curve $G$, which passes through the singularities $\frac{1}{4}(1,1)$ and $\frac{1}{49}(1,20)$. The flip of $G$ produces a surface $X^+$, and a curve $G^+$ (the flip of $G$) which passes through two Wahl singularities. A dot diagram of this transformation is shown in Figure \ref{f3}. For the computation of the Wahl singularities in $X^+$ see \textbf{($<$0)} in \S\ref{enbhd} (before Remark \ref{anti}). In this case the \enii~ is actually an initial \enii. More precisely, we take $m_1=2,a_1=1$ for $[4]$, and $m_2=7,a_2=4$ for $[2,6,2,3]$, and so $\delta=1$, $\Delta=39$, and $\Omega=16$. Then we obtain that $\delta m_1-m_2<0$, and the data for $X^+$ is ${m'}_2=2,{a'}_2=1$ and ${m'}_1=5,{a'}_1=2$.

\begin{figure}[htbp]
\includegraphics[width=9.5cm]{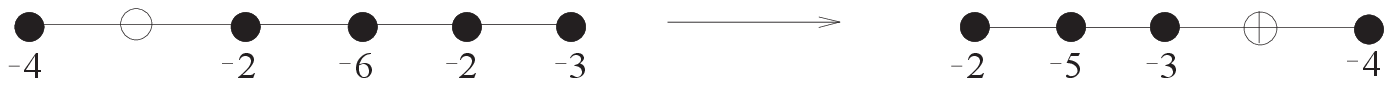}
\caption{A flip}  \label{f3}
\end{figure}

After this flip, the minimal resolution of $X^+$, denoted by $\widetilde{Z}$, is a blow-up of $Z$ $10$ times. The new configuration of relevant
curves is shown in Figure \ref{f4}. Let $W:=X^+$ be the contraction of the configurations $[4]$ ($C$), $[2,2,6]$ ($E_4+E_3+F_1$),
$[2,2,2,7]$ ($A+G_5+G_6+Q$), and $[2,5,3]$ ($E_7+F_2+P$). A standard computation of cohomology groups as in \cite{LP07} (see also \cite[\S4]{PSU} for a concise treatment) shows that $W$ has no local-to-global obstructions; see \S\ref{method1}.

\begin{remark}
In general, if we start with a Lee-Park surface $X$ with no local-to-global obstructions, then (after any birational operation of type \eni~ or \enii~) we end up with $X^+$ with no local-to-global obstructions (including also the divisorial case). This again can be seen via the standard computations in \cite{LP07}.
\end{remark}

\begin{figure}[htbp]
\includegraphics[width=12.5cm]{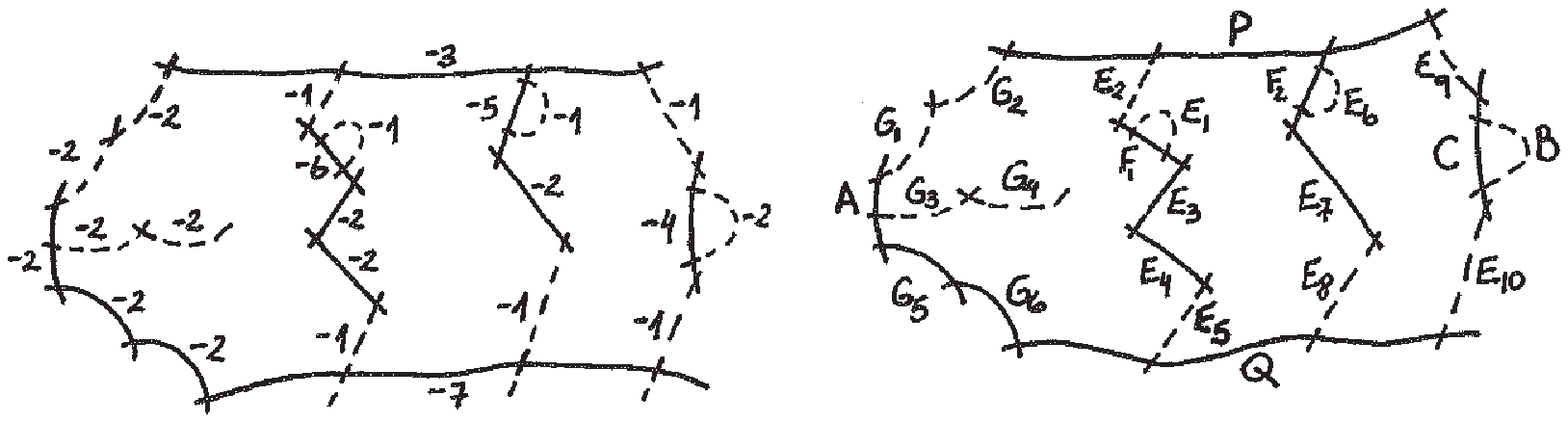}
\caption{Self-intersection (left) and notation (right) for relevant curves in the blow-up $\widetilde{Z}$ of $Z$ $10$ times}  \label{f4}
\end{figure}

We have the $\Q$-numerical equivalence $$K_{\widetilde{Z}} \equiv  -\frac{1}{2} F_1 - \frac{1}{2} F_2 + \frac{1}{2} E_2 + \frac{1}{2}
E_3 + \frac{3}{2} E_4 + \frac{5}{2} E_5 + \frac{1}{2} E_7 + \frac{3}{2} E_8 + E_9 + E_{10},$$ and so, by subtracting the
discrepancies of the singularities in $W$, we verify that the pull-back of $K_W$ can be written as a $\Q$-Cartier divisor with positive coefficients, and it is nef. This is done using the same strategy as in \cite[p.498]{LP07}. Therefore, the general fiber of a $\Q$-Gorenstein smoothing is a smooth minimal projective surface of general type with $K^2=1$, $p_g=0$, and trivial $\pi_1$.

Let us consider its KSBA model $\overline{W}$ (see \S\ref{method1}). Notice first that $\overline{W}$ is not $W$ since $G_4 \cdot K_W=0$. Let $\pi \colon \widetilde{Z} \to W$ be the minimal resolution. The strategy to find $\overline{W}$ will be to identify all curves $\Gamma$ in $\widetilde{Z}$ not contracted by $\pi$, such that $\Gamma \cdot \pi^*(K_W)=0$. In this case we have $\Gamma \cdot K_{\widetilde{Z}}=0$, because of the actual curves in the (effective) support of $\pi^*(K_W)$. Also, since $\Gamma \cdot E_i \neq 0$ may only happen for $i=1$ and $i=6$, we have that $\Gamma \cdot K_{Z'}=0$, where $Z'$ is the blow-up of $Z$ at the nodes of $F_1$ and $F_2$. Notice that $\Gamma$ does not intersects $P$ and $Q$ as well.

Now contract $P$ and $Q$ to obtain a Halphen surface $Z''$ of index $2$ as in Lemma \ref{l1}. In $Z''$ we have $\Gamma \cdot K_{Z''}=0$. But
this means that $\Gamma$ does not intersect a general fiber, and so it is contained in a singular fiber. In this way, the curve $\Gamma$ must
be a smooth rational curve with self-intersection $(-2)$. The elliptic fibration on $Z''$ has three singular fibers: one $I_2^*$
and two $I_2$. The two $I_2$ are $F_1+F_2$ and $B+D$, where $D=\{x_0^2 + 3 x_1^2 = 3 x_2^2 \}$. The two conics
$M=\{x_0^2+3x_1^2=3x_1 x_2 \}$ and $N=\{x_0^2+3x_1^2=-3x_1 x_2 \}$ are part of $I_2^*$, together with $G_4$, $G_3$, $A$, $G_1$, and
$G_5$. Then, we conclude that $\Gamma$ can only be $G_4$, and the KSBA model $\overline{W}$ of $W$ is the contraction
of $G_4$.

Let $\overline{\M}_{1,1}$ be the KSBA moduli space that contains $\overline{W}$. As explained in \S\ref{method2}, locally at $\overline{W}$, this moduli space is the finite quotient of a smooth germ of dimension $8$, and it has $4$ divisors passing through $\overline{W}$ whose general point represents a KSBA singular normal surface with one of the Wahl singularities: $\frac{1}{4}(1,1)$, $\frac{1}{16}(1,11)$, $\frac{1}{25}(1,19)$, and $\frac{1}{25}(1,9)$. As before, we denote the corresponding divisors by $\DD {2 \choose 1}$, $\DD{4 \choose 1}$, $\DD{5 \choose 1}$, and $\DD{5 \choose 2}$. The goal is to identify the smooth minimal model of the surface represented by a general point in $\DD {n \choose a}$ using \S\ref{identify}. For this purpose, we will run MMP on $W$ (instead of $\overline{W}$, see \S\ref{method3}).

\textbf{The general point of $\DD{2 \choose 1}$.} Since there are no local-to-global obstructions to deform $W$, we consider a one
parameter $\Q$-Gorenstein smoothing of all singularities of $W$ except $\frac{1}{4}(1,1)$. In this family, we simultaneously
resolve the singularity $\frac{1}{4}(1,1)$, obtaining a $\Q$-Gorenstein smoothing $(X_0 \subset \mathcal{X}_{0}) \to (0 \in \D)$ of $X_0$, which is $W$ with the singularity $\frac{1}{4}(1,1)$ resolved. The minimal resolution of $X_0$ is $\widetilde{X_0}:=\widetilde{Z}$. In this case we will need only flips, they are shown in Figure \ref{f5}.

We use the dot diagram description in \cite[Notation 5.5]{HTU12}, where in particular $\ominus$ represents the negative curve of the extremal neighborhood, and $\oplus$ represents the flipping positive curve in $X^+$. We remark that in a dot diagram the operations occur in the minimal resolution $\widetilde{X_i}$ of the $X_i$ (see \S\ref{method3}), showing how curves are affected after applying a flip or divisorial contraction.

\begin{figure}[htbp]
\includegraphics[width=7cm]{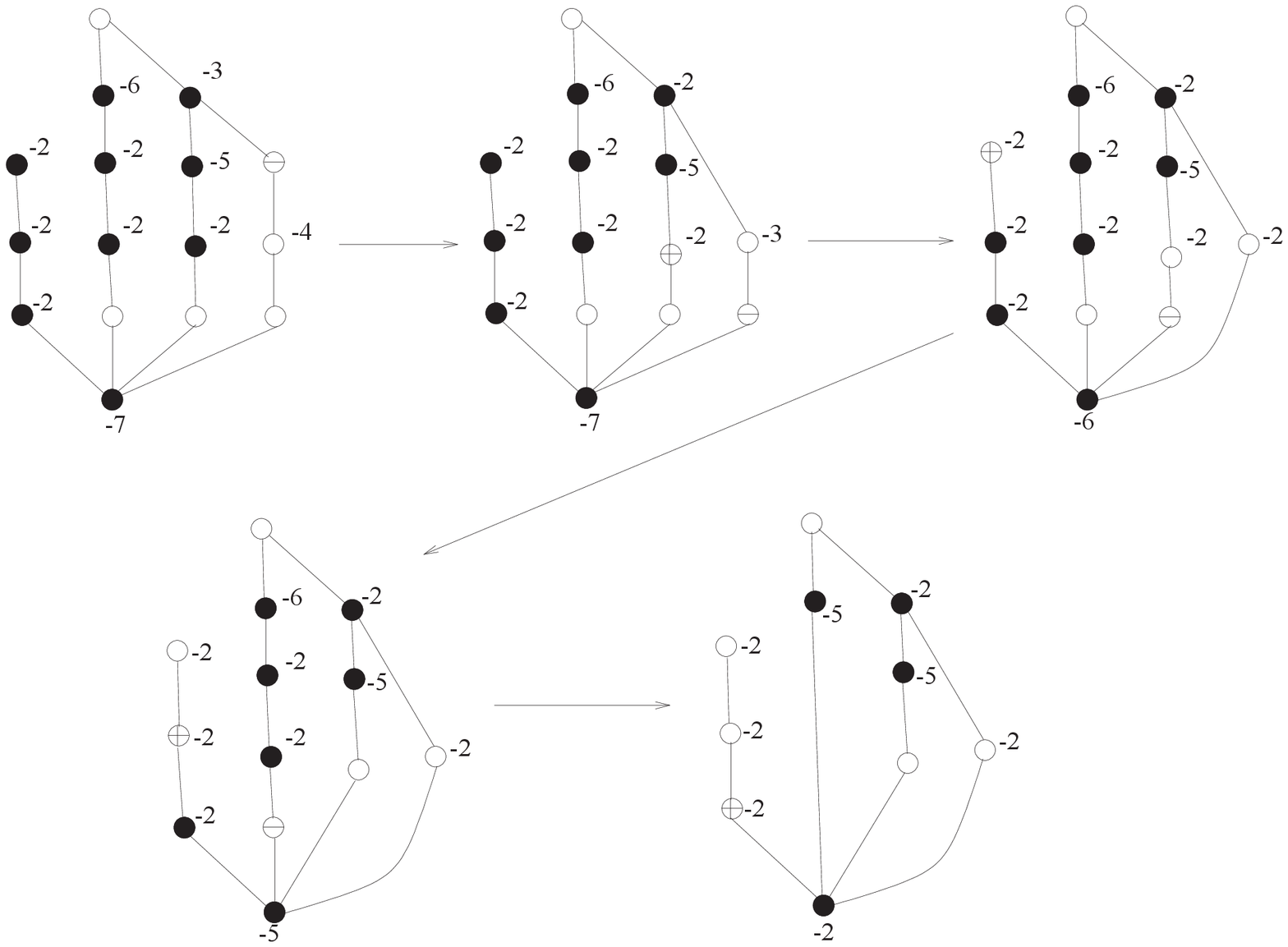}
\caption{Flips for $\DD{2 \choose 1}$}  \label{f5}
\end{figure}

Let $(X_4 \subset \mathcal{X}_{4}) \to (0 \in \D)$ be the final deformation (see \S\ref{method3}). The minimal resolution $\widetilde{X_4}$ of $X_4$ is the blow-up of $Z$ at four points: the nodes of $F_1$ and $F_2$, the intersection of $P$ and $F_1$, and the intersection between $Q$ and $F_2$. The surface $X_4$ is obtained by contracting
$P+F_2$ and $Q+F_1$ in $\widetilde{X_4}$. By Lemma \ref{l1}, we can see $\widetilde{X_4}$ as the blow-up at four points of a Halphen surface of index $2$, and then by Lemma \ref{dolga} with a configuration $[3,3]$, which comes from $[2,5]-1-[2,5]$, we obtain that any $\Q$-Gorenstein smoothing of $X_4$ is a Dolgachev surface
of type $2,3$.

\begin{proposition}
The minimal resolution of a surface representing the general point in $\DD{2 \choose 1}$ is a Dolgachev surface of type $2,3$. It contains a smooth rational curve with self-intersection $(-4)$.
\label{p1}
\end{proposition}

\begin{remark}
Because of the simplicity of $\frac{1}{4}(1,1)$, the previous proposition can also be proved as follows. Let $Y$ be a smooth projective surface containing a $(-4)$-curve $\Gamma$ and $K_Y^2=0$. Let $f \colon Y \to X$ be the contraction of $\Gamma$. If $K_X$ is nef, then $Y$ is not rational. Indeed, if $Y$ is rational, then by Riemann-Roch $h^0(Y,-K_Y) \geq 1$ and so $-K_Y \sim E \geq0$. Since $K_Y \cdot \Gamma = 2$, we have $\Gamma \subset E$. We know that $f^*(2K_X) \sim -2E+\Gamma$. But $E \neq \Gamma$, and so $f^*(2K_X)$ cannot be nef. In this way, in Proposition \ref{p1} we cannot have that the resolution of $\frac{1}{4}(1,1)$ is rational. Also, the Kodaira dimension cannot be $0$ because of $\Gamma$, and it cannot be $2$ because of Proposition \ref{p5}. Therefore it is $1$, and so it has an elliptic fibration. Since it is simply connected, it must have exactly two coprime multiple fibers of multiplicities $a$ and $b$ \cite[II \S3]{D77}. But now it is easy to check using the canonical class formula and $\Gamma$ that the only possibility is $a=2$ and $b=3$, i.e., a Dolgachev surface of type $2,3$.
\end{remark}

\textbf{The general point of $\DD{4 \choose 1}$.} We work as we did with $\DD{2 \choose 1}$, but now with the singularity $\frac{1}{16}(1,11)$. We perform $7$ flips as shown in Figure \ref{f6}. Let $X_7$ be the central singular fiber of the corresponding deformation after the $7$th flip. It has only a $\frac{1}{4}(1,1)$ singularity. The minimal resolution of $X_7$ is the blow up of $Z$ at two points, which are disjoint from the $(-4)$-curve. This situation is as in Theorem \ref{t0} part $(-\infty)$. The general fiber of the $\Q$-Gorenstein smoothing is rational.

\begin{proposition}
The minimal resolution of a surface representing the general point in $\DD{4 \choose 1}$ is a rational
surface with $K^2=-2$. It contains the configuration of rational smooth curves $[6,2,2]$, and a $(-1)$-curve intersecting
the $(-6)$-curve transversally at two points. \label{p2}
\end{proposition}

\begin{figure}[htbp]
\includegraphics[width=7.5cm]{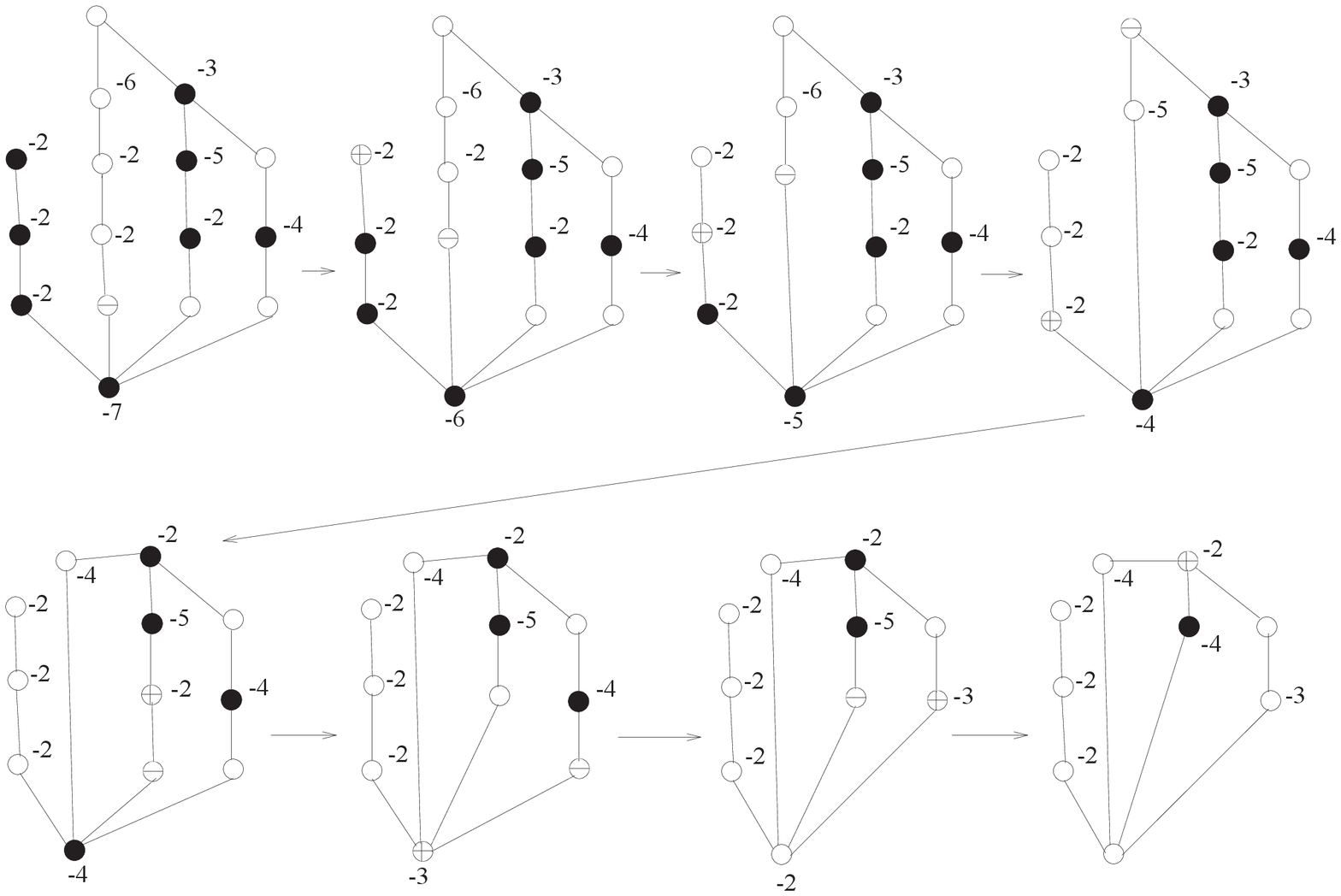}
\caption{Flips for $\DD{4 \choose 1}$} \label{f6}
\end{figure}

The $(-1)$-curve intersecting the $(-6)$-curve transversally at two points comes from the $(-1)$-curve $E_1$ (see Figure \ref{f4}) having the same property in $X_0$. We point out that this $(-1)$-curve does not contain any singularity of $X_0$, and so it lifts in any deformation \cite[IV\S4]{BHPV04}.

\textbf{The general point of $\DD{5 \choose 1}$.} We now perform the sequence of $3$ flips shown in Figure \ref{f7}. Notice that
the situation after the last flip is very similar to the previous case.

\begin{proposition}
The minimal resolution of a surface representing the general point in $\DD{5 \choose 1}$ is a rational surface with $K^2=-3$. It contains the configuration of
rational smooth curves $[7,2,2,2]$, and two disjoint $(-1)$-curves intersecting the $(-7)$-curve transversally at two points each.
\label{p3}
\end{proposition}

\begin{figure}[htbp]
\includegraphics[width=7.5cm]{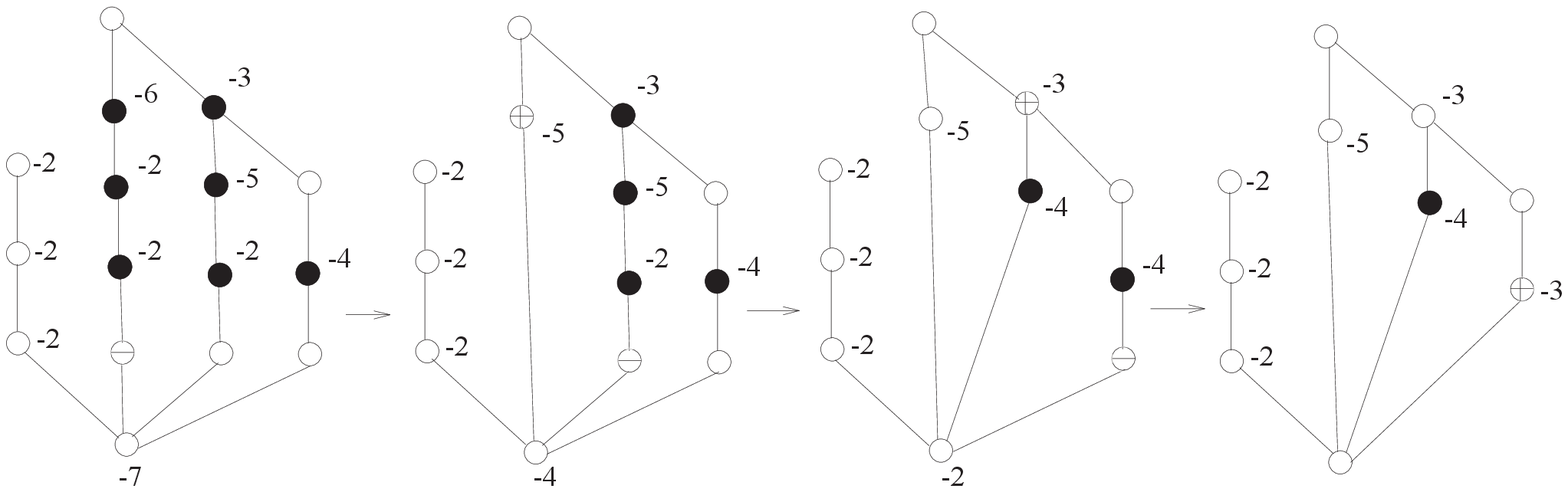}
\caption{Flips for $\DD{5 \choose 1}$} \label{f7}
\end{figure}

The existence of the $(-1)$-curves intersecting the $(-7)$-curve is an application of Proposition \ref{(-1)}, which is applied several times via partial smoothings. We now explain this with no much detail, for a more precise procedure we refer to \cite[\S4]{Urz4}. We start with $X_0$, which is $W$ with $\frac{1}{5^2}(1,4)$ resolved. For notation on curves we refer to Figure \ref{f4}. We first $\Q$-Gorenstein smooth up $\frac{1}{4^2}(1,3)$, and preserve the other singularities of $X_0$ together with the configuration $[7,2,2,2]$. Then the curves $E_2$ and $E_5$ in $X_0$ produce a $(-1)$-curve $E_t$ in the general fiber $Y_1$, intersecting the $(-7)$-curve at one point; we are using Proposition \ref{(-1)}. Notice that $Y_1$ has two singularities, the configuration $[7,2,2,2]$, and the curves $E_t$, $E_8$, $E_9$, and $E_{10}$. We now consider a $\Q$-Gorenstein smoothing of $\frac{1}{4}(1,1)$ keeping the other singularities of $Y_1$ and the configuration $[7,2,2,2]$. By the same proposition we obtain a $(-1)$-curve ${E'}_t$ in the general fiber $Y_2$ from $E_{10}$ and $E_9$. Finally we $\Q$-Gorenstein smooth up $\frac{1}{5^2}(1,9)$ in $Y_2$ to obtain a smooth surface $Y_3$ with the two claimed $(-1)$-curves. Each of them is defined by the pairs $E_t$, $E_8$, and ${E'}_t$, $E_8$, applying again Proposition \ref{(-1)}. These $(-1)$-curves are preserved together with their intersection properties with respect to the $(-7)$-curve, and so we obtain the two $(-1)$-curves in the $\Q$-Gorenstein smoothing of $X_0$.

\textbf{The general point of $\DD{5 \choose 2}$.} In this case we perform the flips shown in Figure \ref{f8}. At the end, the
special fiber is not singular anymore, and so we know that the general fiber of the deformation is a rational surface.

\begin{figure}[htbp]
\includegraphics[width=7.5cm]{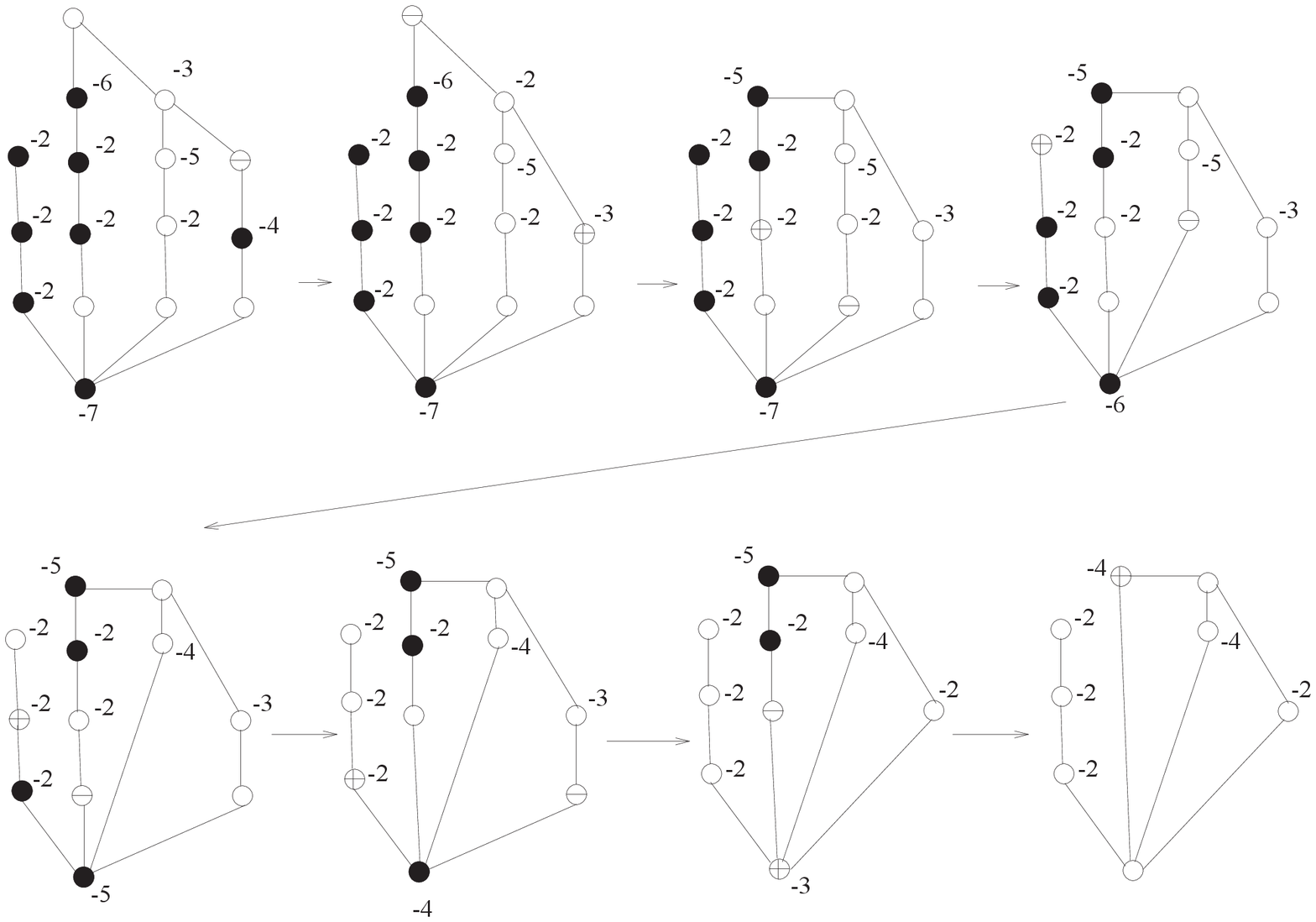}
\caption{Flips for $\DD{5 \choose 2}$} \label{f8}
\end{figure}

\begin{proposition}
The minimal resolution of a surface representing the general point in $\DD{5 \choose 2}$ is a rational surface with $K^2=-2$. It contains the configuration of rational smooth curves $[2,5,3]$, and a $(-1)$-curve intersecting the $(-5)$-curve transversally at two points. \label{p4}
\end{proposition}

The $(-1)$-curve comes from the $(-1)$-curve $E_6$ intersecting the $(-5)$-curve transversally at two points. This finishes the description of the ``general" KSBA neighbors of $\overline{W}$.

\begin{remark}
We can construct a stable surface $T$ with the same Wahl singularities as $W$ by using a more general elliptic rational surface, which has singular fibers $I_4 + 6 I_1 + I_2$. This elliptic fibration has moduli dimension $4$. From the $4$ Wahl singularities $\frac{1}{4}(1,1)$, $\frac{1}{16}(1,3)$, $\frac{1}{25}(1,4)$, and $\frac{1}{25}(1,9)$ of $T$, we obtain the other $4$ dimensions for the moduli space around $T$, completing the $8$ dimensions needed (see \S\ref{method1}).
\end{remark}

\begin{remark}
For the other example with $K^2=1$ in \cite[Fig.6]{LP07}, we have a surface with Wahl singularities and canonical class nef. This example is related to the previous in the following way. Take a $(-1)$-curve from \cite[Fig.6]{LP07} between the configuration $[2,2,6]$ and $[4]$ (there are two choices). The configuration
$[2,2,6]-1-[4]$ represents the data of an extremal P-resolution (Definition \ref{P-res}) of $\frac{1}{36}(1,13)$. But this singularity
admits another extremal P-resolution, which is $[3,5,2]-2$. (We recall that \cite[\S4]{HTU12} is a section devoted to singularities having two extremal P-resolutions.)
Now consider the corresponding $\Q$-Gorenstein smoothing of the new surface (which has only Wahl singularities). The canonical class of the central fiber is not nef, because there is a $(-1)$-curve intersecting the $(-8)$-curve at one point. So we perform one flip of type \eni. After that, the resulting surface is the previous example. Therefore, we have a sort of dual families related by $\frac{1}{36}(1,13)$. These two families are different, they are located around two different stable surfaces of the moduli space. This is a common ``wormhole" phenomena in Lee-Park type of examples, which comes from the fact that a given cyclic quotient singularity may have two extremal P-resolutions (and no more, see \cite[\S4.2]{HTU12}).

The analog results for partial smoothings of the Wahl singularities in the example \cite[Fig. 6]{LP07} are: for both $\frac{1}{4}(1,1)$ we obtain Dolgachev surfaces of type $2,3$ (for $\frac{1}{8}(1,3)$ we also have Dolgachev surfaces of the same type), and for the other singularities we obtain rational surfaces.
\label{rrr}
\end{remark}

By Proposition \ref{p5}, we know that the KSBA boundary appearing (in this way) for $K^2=1$ consists of surfaces whose minimal resolution is not of
general type. This is not the case for $K^2>1$, as we will see in the next sections.

\section{$K^2=2$} \label{s2}

In this section and the next, the proof that $X_n$ (final surface after certain birational operations, see \S\ref{method3}) has nef canonical class can be done explicitly, using the strategy in \cite[p.498]{LP07}. As we did in \S\ref{s1}, we will omit those computations.

\begin{figure}[htbp]
\includegraphics[width=11cm]{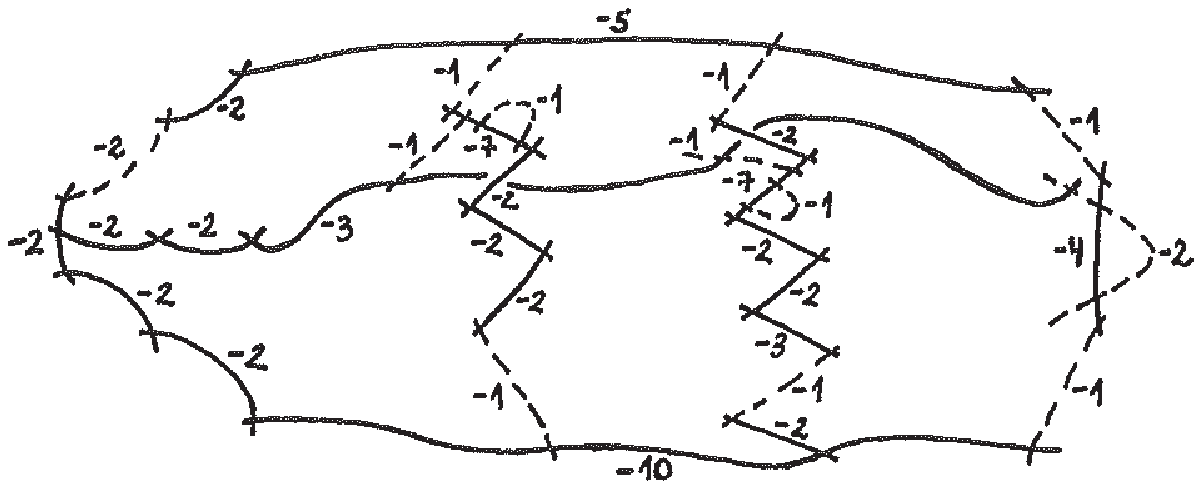}
\caption{The example \cite[Fig.2]{LP07}} \label{f9}
\end{figure}

Let us take the example in Figure 2 of \cite{LP07}. It starts with the same elliptic fibration used in \S\ref{s1}. The corresponding surface $W$ with only Wahl singularities has $K_W$ nef. One can use Lemma \ref{l1} to show that $K_W$ is ample in this case, so $W=\overline{W}$ is already a stable surface. The five Wahl singularities define five boundary divisors. We label them as before: $\DD{2 \choose 1}$ for $[4]$, $\DD{3 \choose 1}$ for $[2,5]$, $\DD{5 \choose 1}$ for $[7,2,2,2]$, $\DD{9 \choose 4}$ for $[2,7,2,2,3]$, and $\DD{15 \choose 7}$ for $[2,10,2,2,2,2,2,3]$.

\textbf{The general point of $\DD{2 \choose 1}$.} We proceed as in \S\ref{s1}. We perform the $4$ flips shown in Figure \ref{f10}. The first two are \eni ~flips, the last two are \enii ~flips. The last singular surface $X_4$ has five Wahl singularities. The canonical divisor $K_{X_4}$ is nef, and $K_{X_4}^2=1$.

\begin{figure}[htbp]
\includegraphics[width=8.5cm]{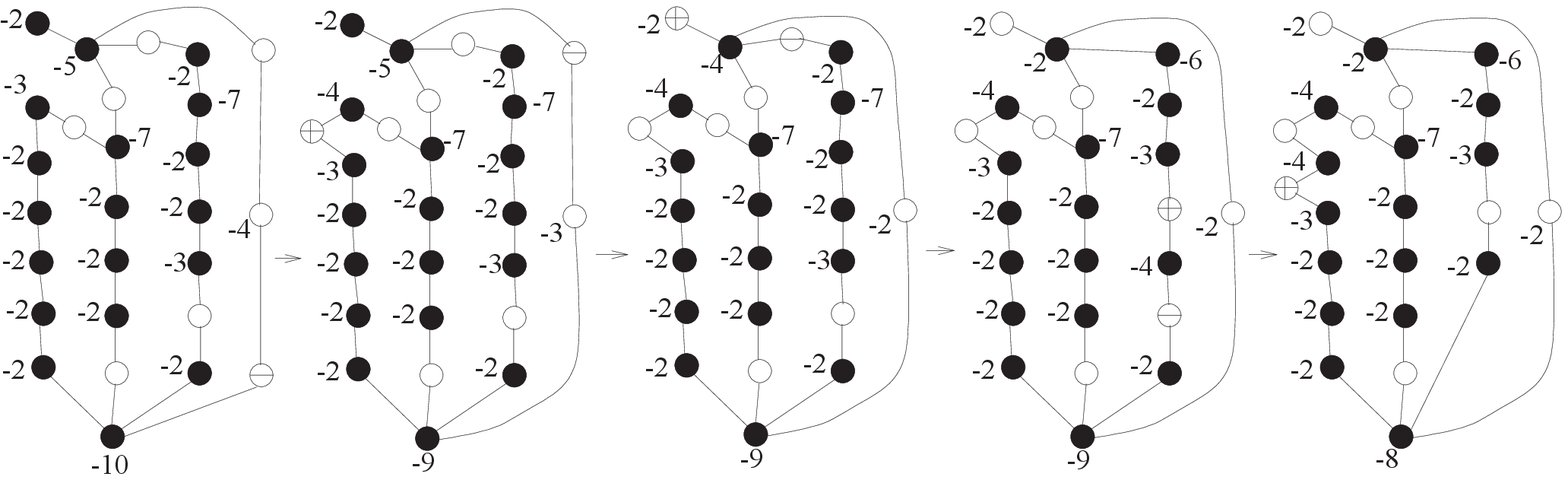}
\caption{Flips for $\DD{2 \choose 1}$} \label{f10}
\end{figure}

\begin{proposition}
The minimal resolution of a surface representing the general point in $\DD{2 \choose 1}$ is a simply connected surface of general type with $p_g=0$ and $K^2=1$. It contains a $(-4)$-curve. \label{p6}
\end{proposition}

This proposition gives a new example $X_4$ with $K^2=1$ (Lee-Park type). Its minimal resolution has T-configurations $[4]$, $[4]$, $[2,6,2,3]$, $[7,2,2,2]$, and $[3,2,2,2,8,2]$ (see Remark \ref{newfam}).

\begin{figure}[htbp]
\includegraphics[width=8.5cm]{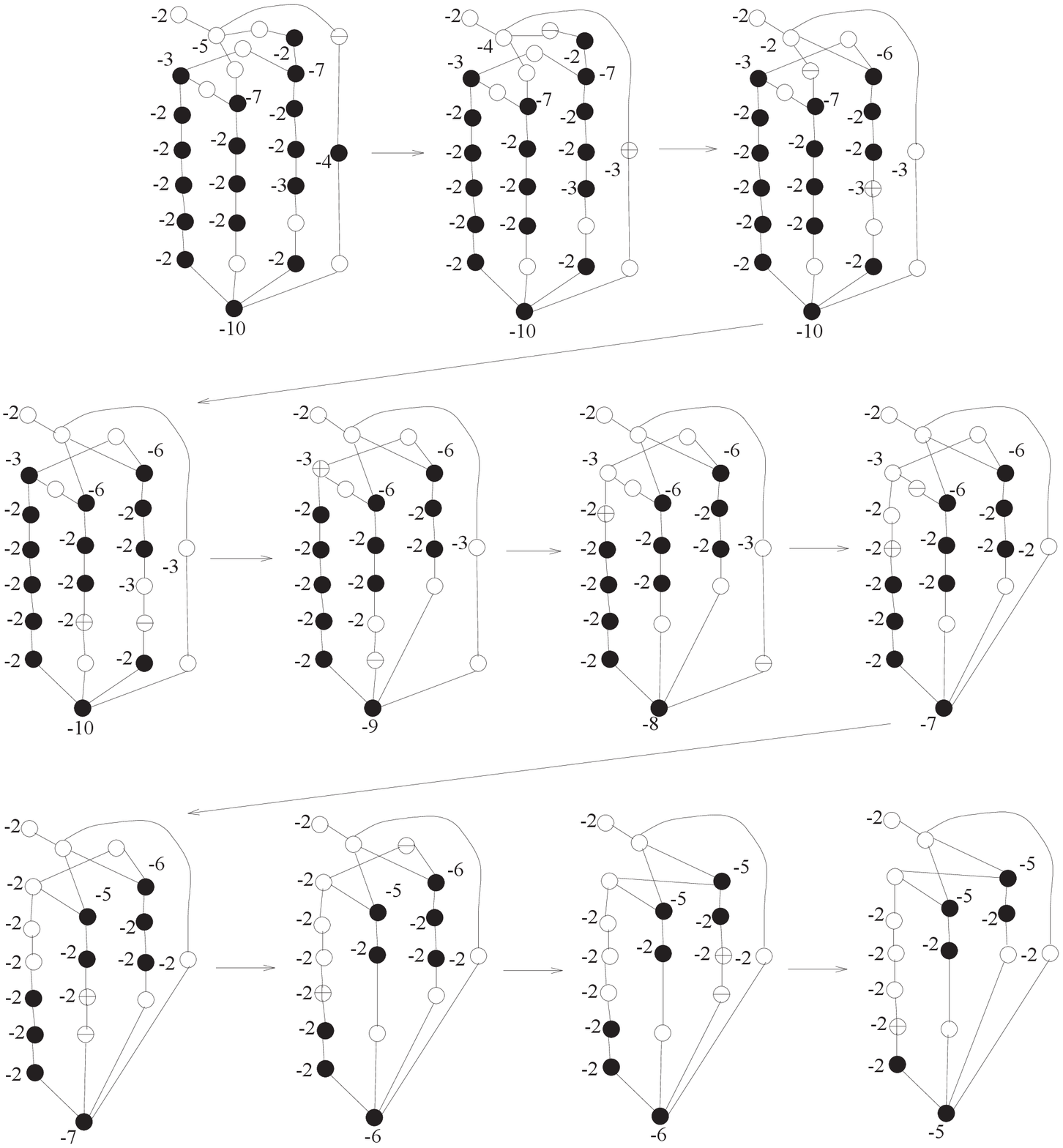}
\caption{Flips for $\DD{3 \choose 1}$} \label{f11}
\end{figure}

\textbf{The general point of $\DD{3 \choose 1}$.} Here we perform the $10$ flips shown in Figure \ref{f11}. One can verify that $X_{10}$, the last surface, has $K_{X_{10}}^2=0$ and $K_{X_{10}}$ nef. Therefore, the general fiber of the $\Q$-Gorenstein smoothing is a Dolgachev surface of some type $n_1,n_2$, since we already know that it is simply connected. One way to find $n_1,n_2$ is by arguing that a $\Q$-Gorenstein smoothing of $X_{10}$ was used in the second example with $K^2=1$ (Remark \ref{rrr}). There we knew that the Dolgachev surface contained a $(-4)$-curve, and so one obtains $n_1=2$, $n_2=3$. So we have same multiplicities for our current example (although we do not know if there is a $(-4)$-curve inside).

\begin{proposition}
The minimal resolution of a surface representing the general point in $\DD{3 \choose 1}$ is a Dolgachev surface of type $2,3$ which
contains a configuration $[2,5]$. \label{p7}
\end{proposition}

For the other $3$ divisors we perform certain flips to deduce that its general point is rational, and the minimal resolution has

\begin{itemize}
\item[$\DD{5 \choose 1}$:] $K^2=-2$ with a configuration $[2,2,2,7]$ inside.

\item[$\DD{9 \choose 4}$:] $K^2=-3$ with a configuration $[3,2,2,7,2]$ inside.

\item[$\DD{15 \choose 7}$:] $K^2=-6$ and a configuration $[3,2,2,2,2,2,10,2]$ inside.
\end{itemize}

For the other example in \cite{LP07}, i.e. \cite[Fig.4]{LP07}, we find the following. For each of the $\frac{1}{4}(1,1)$ singularities one obtains a simply connected surface of general type with $K^2=1$ and $p_g=0$. If we keep both singularities $\frac{1}{4}(1,1)$, then one obtains a Dolgachev surface $2,3$ with two disjoint $(-4)$-curves. Finally, for each of the other Wahl singularities one obtains rational surfaces.

\section{$K^2=3$} \label{s3}

In \cite{PPS09} there are five examples producing simply connected surfaces of general type with $p_g=0$ and $K^2=3$. We take the one in \cite[Fig.8]{PPS09} because, as explained in \cite{PPS09e}, it contains a negative curve which makes the canonical divisor of the singular surface not nef. This curve gives the data of a flipping \enii. The flip is shown in Figure \ref{f12}.

\begin{figure}[htbp]
\includegraphics[width=10cm]{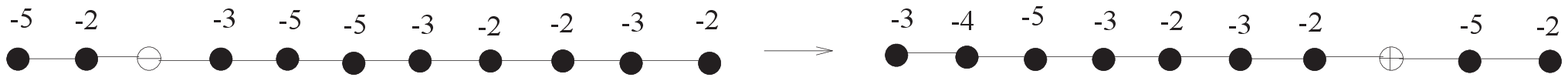}
\caption{Flip for \cite[Fig.8]{PPS09}} \label{f12}
\end{figure}

We can show that after this flip, the resulting surface $W$ has nef canonical divisor via the \cite{LP07} strategy. Hence this example has indeed the claimed
properties in \cite{PPS09}. The minimal resolution $\widetilde{W}$ of $W$ is in Figure \ref{f13}. Let $F$ be the general fiber of the induced elliptic fibration on $\widetilde{W}$. Then, following the notation in Figure \ref{f13}, we have $$K_{\widetilde{W}} \sim \sum_{i=1}^{15} E_i + E_7 + 2 E_8
+ E_{11} + E_{13} + E_{15} - F$$ and so $K_{\widetilde{W}} \equiv - \frac{1}{2} F_1 - \frac{1}{2} F_2+ E_1 + E_2 + \frac{1}{2} E_4 +
\frac{1}{2} E_5 + \frac{1}{2} E_7 + E_8 + \frac{1}{2} E_9 + E_{10} + 2 E_{11} + E_{12} + 2 E_{13} + E_{14} +2 E_{15}$. After we subtract the
discrepancies, we obtain an effective $\Q$-divisor for $\sigma^*(K_{W})$. It is easily verified that it is nef by intersecting it with the curves in its support.

\begin{figure}[htbp]
\includegraphics[width=13cm]{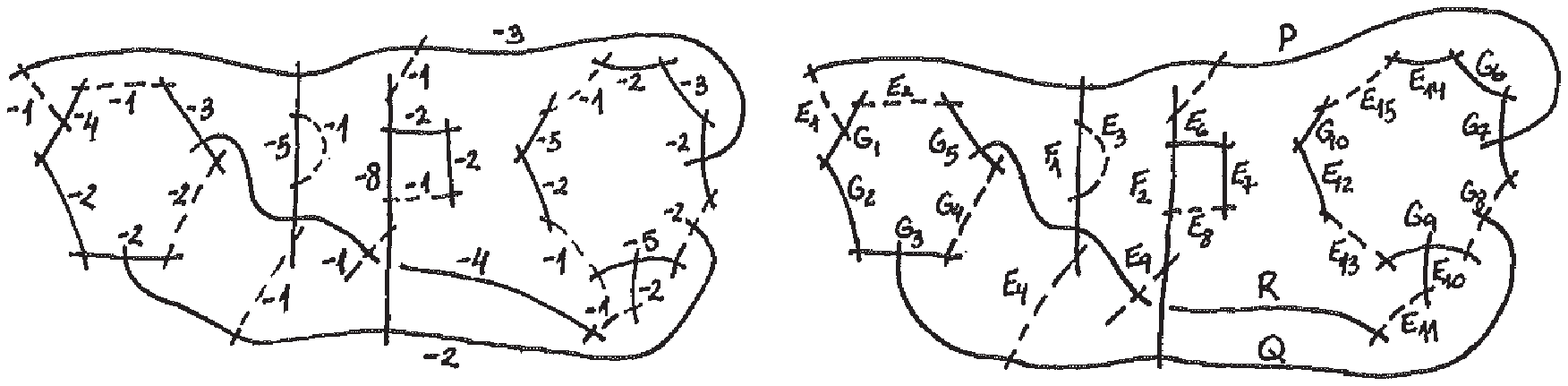}
\caption{$\widetilde{W}$ and relevant curves} \label{f13}
\end{figure}

Moreover, this support contains $E_5$, $F_2$, $E_6$, $E_7$, $E_8$, and $E_9$ which is the support of a fiber. This implies that the
only curves which could have intersection $0$ with $K_W$ are components of fibers. Then, the only one is $E_{13}$. Let
$\overline{W}$ be the contraction of $E_{13}$, so $K_{\overline{W}}$ is ample and $\overline{W}$ is a stable surface. The corresponding point in the moduli space is a finite quotient of a smooth germ of dimension $4$. The singularities of $\overline{W}$ are $\frac{1}{30^2}(1,30 \cdot 11-1)$, $\frac{1}{2 \cdot 3^2}(1, 2
\cdot 3 \cdot 1 -1)$, and $\frac{1}{16^2}(1,16 \cdot 11-1)$. Their $\Q$-Gorenstein smoothings give precisely the dimension
$4=1+2+1$. In that sense, this surface $\overline{W}$ is a ``maximal degeneration".

The loci in the moduli space defined by keeping the singularity $\frac{1}{18}(1, 5)$ has codimension
$2$. We have that the minimal model of a resolution of general point in this loci is a simply connected surface of general type
with $K^2=1$ (and $p_g=0$), with a configuration $[4,3,2]$ inside. If we $\Q$-Gorenstein deform the singularity $\frac{1}{18}(1,5)$ into $\frac{1}{9}(1,2)$, and $\Q$-Gorenstein smooth up all the other singularities, then we obtain a surface of general type with $K^2=1$. Finally, for each of the other two singularities we have divisors parametrizing rational surfaces. This describes the general points of the associated divisors $\DD{30 \choose 11}$, $\DD{3 \choose 1}$, and $\DD{16 \choose 11}$.

\begin{remark}
With the example \cite[Fig.9]{PPS09} we can show that there are $K^2=2$ surfaces of general type with $p_g=0$ in the boundary of
the moduli space for $K^2=3$. We keep in a $\Q$-Gorenstein deformation the singularity $\frac{1}{4}(1,1)$ and smooth up the
other two. After some flips we obtain a singular surface with $4$ Wahl singularities whose exceptional configurations are
$[2,3,2,3,5,4,3]$, $[2,5]$, $[2,5]$, and $[6,2,2]$. Its canonical class is nef and $K^2=2$.
\end{remark}





\vspace{0.3cm}

{\small Facultad de Matem\'aticas,

Pontificia Universidad
Cat\'olica de Chile,

Santiago, Chile.}

\end{document}